\documentclass[notitlepage,11pt,reqno]{amsart}
\usepackage{color}
\usepackage{amssymb,nicefrac,cite,bm,upgreek}
\usepackage[mathscr]{eucal}

\definecolor{dmagenta}{rgb}{.4,.1,.5}
\definecolor{007}{rgb}{.0,.0,.7}
\definecolor{dred}{rgb}{.5,.0,.0}

\definecolor{dgreen}{rgb}{.0,.5,.0}
\definecolor{dblue}{rgb}{.0,.0,.5}

\definecolor{violet}{rgb}{.3,.0,.9}
\definecolor{orange}{cmyk}{0,.5,.1,.0}
\definecolor{dcyan}{cmyk}{.5,.0,.0,.0}
\definecolor{dyellow}{cmyk}{.0,.0,.5,.0}
\definecolor{cm}{cmyk}{1,.0,.0,.0}
\numberwithin{equation}{section}
\allowdisplaybreaks

\newtheorem{theorem}{Theorem}[section]
\newtheorem{lemma}{Lemma}[section]

\theoremstyle{definition}
\newtheorem{definition}{Definition}[section]

\newtheorem{example}{Example}[section]

\theoremstyle{remark}
\newtheorem{remark}{Remark}[section]

\newcommand{\eG}{{\mathscr{G}}}

\newcommand{\eD}{\mathfrak{D}}

\DeclareMathOperator{\Exp}{\mathbb{E}}
\DeclareMathOperator{\Prob}{\mathbb{P}}
\newcommand{\D}{\mathrm{d}}

\newcommand{\R}{\mathbb{R}}
\newcommand{\B}{\mathbb{B}}

\newcommand{\NN}{\mathbb{N}}
\newcommand{\Act}{\mathbb{U}}

\newcommand{\Ind}{\mathbf{1}}

\newcommand{\calA}{\mathcal{A}}
\newcommand{\calB}{\mathcal{B}}
\newcommand{\calC}{\mathcal{C}}

\newcommand{\calF}{\mathcal{F}}
\newcommand{\calH}{\mathcal{H}}
\newcommand{\calK}{\mathcal{K}}
\newcommand{\calL}{\mathcal{L}}
\newcommand{\calM}{\mathcal{M}}
\newcommand{\calP}{\mathcal{P}}

\newcommand{\calV}{\mathcal{V}} 
\newcommand{\calX}{\mathcal{X}} 
\newcommand{\calY}{\mathcal{Y}}

\newcommand{\al}{\alpha}

\newcommand{\eps}{\varepsilon}

\newcommand{\del}{\delta}

\newcommand{\cast}{\circledast}

\newcommand{\bvnorm}[1]{[\kern-0.45ex[\kern0.1ex #1 \kern0.1ex]\kern-0.45ex]}
\newcommand{\Uadm}{\mathfrak{U}}
\newcommand{\Usm}{\mathfrak{U}_{\mathrm{SM}}}
\newcommand{\Udsm}{\mathfrak{U}_{\mathrm{DSM}}}

\newcommand{\abs}[1]{\lvert#1\rvert}
\newcommand{\norm}[1]{\lVert#1\rVert}

\newcommand{\sorder}{{\mathfrak{o}}} 

\newcommand{\df}{:=}

\let\oldtocsection=\tocsection
\let\oldtocsubsection=\tocsubsection
\let\oldtocsubsubsection=\tocsubsubsection
\renewcommand{\tocsection}[2]{\hspace{0em}\oldtocsection{#1}{#2}}
\renewcommand{\tocsubsection}[2]{\hspace{1em}\oldtocsubsection{#1}{#2}}
\renewcommand{\tocsubsubsection}[2]{\hspace{2em}\oldtocsubsubsection{#1}{#2}}

\begin{document}

\title[MFG with ergodic cost]
{Mean Field Games with Ergodic cost for Discrete Time Markov Processes}

\author{Anup Biswas}
\address{Department of Mathematics,
Indian Institute of Science Education and Research, 
Dr. Homi Bhabha Road, Pune 411008, India.}
\email{anup@iiserpune.ac.in}

\date{\today}
\thanks{This research of the author was supported in part by a INSPIRE faculty fellowship.}

\begin{abstract}
We consider mean field games with ergodic cost in the framework of a general discrete time controlled Markov processes. The state space of
the processes is given by a general $\sigma$-compact Polish space.
Under certain conditions, we show the existence of a mean field game equilibrium.  We also study the 
$N$-person game where the players interacts with each other via their empirical measure. We show that the $N$-person game has Nash equilibrium and as $N$ tends to infinity the equilibria converge to a mean field game solution.
\end{abstract}



\maketitle

\section{\bf Introduction}
The goal of this article is to give a general framework of  discrete time controlled Markov processes for the analysis
of mean field games (MFG) and its connection with $N$-person games. 
Intuitively, mean filed game can be described as follows: there is a single representative agent whose state dynamics is also effected by an environment distribution coming from the infinite number of agents. State process of the representative agent can not influence the environment while solving his/her own optimization problem. Since all agents are indentical and act in the same way, the distribution of the representative agent should agree with the environment distribution.
Mean field games are introduced independently, in the seminal work of Huang, Malham\'{e} and Caines \cite{hmc} and in the seminal work of
Lasry and Lions \cite{lasry-lions} where the game dynamics are given by a certain class of large-population
stochastic differential games. In such games, there are $N$ number of (indistinguishable) players trying to optimize
certain cost that depends on the decisions of other players. Since the objectives are coupled, one naturally look for
Nash equilibrium.If the number of agents $N$ tends to infinity one expects the Nash equilibria to converge to a 
MFG solution. 

There are three major questions of interest in MFG: (1) existence of MFG solution, (2) uniqueness of MFG solution, and
(3) establishing regorous connection with the $N$-person games. Either one tries to show that the $N$-person game Nash equilibrium tends to the MFG equilibrium (\cite{lasry-lions, anup-ari, feleqi, fisher, bardi-priuli}) or tries to construct
approximate Nash equilibrium for $N$-person games from MFG solutions (\cite{hmc}). 
Question of existence has been resolved 
for a large class of problems (where the state dynamics are given by controlled diffusions). Uniqueness results for a general class seems difficult and only available under some monotonicity condition (see Remark~\ref{R-uni}
below).

Let us now give a quick survey on MFG literature. MFG is a fast growing field. We refer the readers to
the survey articles \cite{cardaliaguet, gomes-saude} and the book \cite{ben-fre-yam} for developments in MFG. A short 
account of applications of MFG is given at \cite{gu-lasry-lions}. For finite horizon type costs, existence of MFG
solution is established in \cite{carmona-delarue} using stochastic maximum principle whereas \cite{lacker} considers
the set up of controlled martingale problem and shows existence of MFG solution. Convergence of Nash equilibria of the $N$-person games to the solution of MFG  is obtained in \cite{fisher} where the cost function also has finite time horizon. Ergodic type costs are considered in \cite{lasry-lions, feleqi, anup-ari, bardi-priuli}. In \cite{lasry-lions, feleqi}, the controlled diffusions take values in a compact metric space whereas in \cite{anup-ari, bardi-priuli} the state space is the Euclidean space. Convergence of equilibria of $N$-person game is also discussed in \cite{lasry-lions, feleqi, anup-ari, bardi-priuli}. All the above mentioned work are done under the settings of controlled stochastic differential equations. MFG for finite state processes are studied in \cite{gomes-mohr-souza, gomes-2, kolokoltsov, gueant}.

In this article, we consider a general class of discrete Markov processes taking values in a $\sigma$-compact Polish space and the cost function is given by an ergodic cost. We also impose a blanket (geometric) stability condition on the processes (see \eqref{01}). This ensure that every (stationary) controlled Markov process has an invariant measure. Under a set of general conditions we show the existence of MFG solution. The proof technique uses
Kakutani-Fan-Glicksberg fixed point theorem. To apply this theorem one needs to show that the set valued map under consideration is upper-hemicontinuous (Lemma~\ref{L4.3}). This is done by showing equicontinuity property
of the value functions for the ergodic control problems. When the controlled states are governed by non-degenerate diffusion one can use Haranack's inequality (see \cite{anup-ari}) to establish such estimate. In our setup, we use the
minorization condition (see \eqref{02}) and the construction of split chain to obtain equicontinuity. We also consider the $N$-person game and using Kakutani-Fan-Glicksberg fixed point theorem we establish existence of Nash equilibrium. Under an additional convexity assumption (condition (A7)) we show in Theorem~\ref{Thm-4}
that the Nash equilibria converges, as $N$ tends to infinity, to the solution of MFG. 

Therefore to summerize our main contributions in this article:
\begin{itemize}
\item[--] we consider the MFG with ergodic cost for a general class of controlled discrete time Markov chains and show existence of MFG solution;
\item[--] we establish existence of $N$-person game Nash equilibrium;
\item[--] we derive the convergence of the Nash equilibria to the MFG solution.
\end{itemize}
The article is organized as follows. In Section~\ref{S-model} we introduce our controlled Markov chain and the
set of assumptions that we impose on our processes. Section~\ref{S-MFG} states our main existence result on
MFG solution. In Section~\ref{S-Nash} we discuss the $N$-person game and prove existence of Nash equilibrium.
Then in Section~\ref{S-convergence} we prove convergence of Nash equilibria to MFG solution. Finally, in 
Section~\ref{Ergodic HJB} we give the proof of Theorem~\ref{Thm-1}.

\subsection{Notation}
The set of nonnegative real numbers is denoted by $\R_{+}$,
$\NN$ stands for the set of natural numbers. The interior, closure, the boundary and the complement
of a set $A\subset\calX$, $\calX$ is a topological space, are denoted
by $A^o$, $\overline{A}$, $\partial{A}$ and $A^{c}$, respectively. 
$\Ind_A$ denotes
the indicator function of the set $A$.
The open ball of radius $R$ around $0$ is denoted by $B_{R}$.
Given two real numbers $a$ and $b$, the minimum (maximum) is denoted by $a\wedge b$ 
($a\vee b$), respectively.
Define $a^{+}\df a\vee 0$ and $a^{-}\df-(a\wedge 0)$. 
By $\delta_{x}$ we denote the Dirac mass at $x$. 

Given any Polish space $(\calX, \D_\calX)$, we denote by $\calP(\calX)$ the set of
probability measures on $\calX$ and $\calM(\calX)$ the set of all bounded signed measure on $\calX$.
For $\mu\in\calP(\calX)$ and a Borel measurable map $f\colon\calX\to\R$,
we often use the abbreviated notation
$$\mu(f)\;\df\; \int_{\calX} f\,\D{\mu}\,.$$
The total variation norm of a measure $\mu$ is denoted by $\norm{\mu}_{TV}$.
The Borel $\sigma$-algebra on $\calX$ is denoted by $\calB(\calX)$. $\calC(\calX)$ ($\calC_b(\calX)$) denotes the set of all real valued (bounded) continuous
functions of $\calX$.
Law of a random variable $X$ is denoted by $\calL(X)$.
Also $\kappa_{1},\kappa_{2},\dotsc$ 
are used as generic constants whose values might vary from place to place.

\section{\bf Model and Assumptions}\label{S-model}

Let $(\calX, \D_\calX)$ be a Polish space which will be treated as the state space for our controlled Markov processes. We shall assume that $\calX$ is $\sigma$-compact i.e., $\calX$ can be written
as a countable union  of compact subsets of $\calX$. Let $(\Omega, \calF, \Prob)$ be the probability space on which  the random variables are defined. $\Act$ is a compact metric space
that denotes the control space for the Markov controls. The controlled stochastic kernels are given by the Borel-measurable map $P(dy| \cdot, \cdot): \calX\times\Act\to \calP(\calX)$. A Markov control is given by a
collection of Borel-measurable maps $u_n\colon \calX\to\Act$, $n\in\NN$. Therefore given a Markov control $\{u_n\}$, the controlled Markov process $\{X_n\}$ is defined as follows: If $X_i = x\in\calX$, then the possible location
of $X_{i+1}$ is determined by the distribution $P(\cdot | x, u_i(x))$. The set of all Markov controls are denoted by $\Uadm$. A Markov control is called stationary if $u_n= u$, for all $n\in\NN$, for some Borel-measurable map $u\colon \calX\to\Act$. Let $\Usm$ be the set of all stationary Markov controls.

A lower-semicontinuous (lsc) function $g:\calX\to\R$ is said to be inf-compact if $\{x\in\calX\; :\; g(x)\leq \kappa\}$ is compact subset of $\calX$ for all $\kappa\in\R$. 
Let $\{K_n\}$ be a collection of compact subsets in $\calX$ having following property: if $\calX$ is compact
then $K_n=\calX$, otherwise
\begin{equation}\label{0001}
K_n\subset \; K^o_{n+1}, \quad \text{and}\quad \cup_{n} K_n \; =\; \calX.
\end{equation}
For a non-negative, lsc function $g\colon\calX\to\R_+$, we denote by $\sorder(g)$ the collection of function $f\colon \calX\to\R$ that satisfies
$$\limsup_{n\to\infty}\; \sup\Big\{\frac{\abs{f(x)}}{1+ g(x)}: x\in\calX\setminus K_n\Big\}=0.$$
It should be observed that for non-compact $\calX$,
 $\sorder(g)$ does not depend on the choice of family $\{K_n\}$ satisfying \eqref{0001}. This is due to the fact that
 $\cup K^o_n=\calX$.
 If $\calX$ is compact, we have $K_n=\calX$ and the supremum of empty set would be assumed to be $0$.
 We shall assume that
\begin{itemize}
\item[{\bf (A1)}] there exists lsc, inf-compact function $\calV\colon \calX\to\R, \; \calV\geq 1$, and a compact set $C\subset \calX$, satisfying
\begin{equation}\label{01}
\sup_{u\in\Act} \int_{\calX} \calV(y) \, P(dy| x, u) -\calV(x) \; \leq - \beta_1\calV(x)\,\Ind_{C^c}(x) + \beta_2\, \Ind_C(x),
\end{equation}
where $\beta_1, \beta_2$ are positive constants. Moreover, $\inf_{x\in C^c} \calV(x)\geq (\sup_{x\in C}\calV(x)\vee 2\frac{\beta_2}{\beta_1})$.
\end{itemize}
\eqref{01} is a standard stability condition and often used to obtain exponential convergence 
of the transition probabilities to the invariant measure (see \cite{meyn-tweedie}). 
We also need additional hypothesis on the controlled Markov processes to ensure that it have invariant measures. Let us first introduce randomized policy which helps to establish
convexity property of the set of invariant measures. A randomized (or relaxed) Markov control is given by the collection of
Borel-measurable maps $u_n:\calX\to \calP(\Act)$. We can extend the map of transition probability measures on $\calP(\Act)$ as follows
$$\int_{\calX\times \Act} f(y, u) P(dy | x, v(x))\; \df\; \int_{\calX\times \Act} f(y, u) P(dy | x, u) v(du| x),$$ 
where $f:\calX\times\Act\to\R$ is a bounded, Borel-measurable map. Hence we can also define controlled Markov process associated
to these randomized Markov controls. Given a Markov control $v\in\Uadm$, we denote the corresponding controlled Markov chain as $\{X_n(v)\}_{n\geq 0}$. 
We note that the elements in $\Uadm$ can be seen as randomized Markov controls where 
$u_n(x)\df \delta_{u_n(x)}$. We denote the set of all Markov controls (including randomized) by $\Uadm$
and by $\Usm$ we denote the set of all stationary (including randomized) Markov controls. A stationary Markov policy 
given by the map $u:\calX\to\Act$ is called deterministic stationary Markov control. The set of all deterministic stationary Markov controls
is denoted by $\Udsm$. Thus we have $\Udsm\subset\Usm\subset\Uadm$. One can often endow the space $\Usm$ with a topology that renders it a
compact metric space. We assume that
\begin{itemize}
\item[{\bf (A2)}]there exists a metric $\D_M$ on $\Usm$ such that $(\Usm, \D_M)$ becomes a compact metric space. Also, if
$v^n(x)\to v(x)$ as $n\to\infty$, for $x\in\calX$, for some stationary Markov controls $v^n, v\in\Usm$, then $\D_M(v^n, v)\to 0$ , as $n\to\infty$.
Furthermore, for every bounded, continuous $f:\calX\times\Act\to\R$, $k\in\NN$, the map
$(x, v)\in\calX\times\Usm\mapsto (\Exp_x[f(X_1, v(X_1))], \ldots, \Exp_x[f(X_k, v(X_k))])$ is a continuous.
\end{itemize}
Let us remark that the above hypothesis (A2) is quiet natural and generally satisfied by large class of Markov processes. If there is $\sigma$-finite, non-negative, regular Borel measure $\iota$ on $\calX$, we can define
the metric $\D_M$ on $\Usm$ viewing it as a subset of the unit ball in $L^\infty (\calX, \calM(\Act))$ with respect
to the measure $\iota$. Note that the unit ball of $L^\infty (\calX, \calM(\Act))$ is compact (by Banach-Alaoglu theorem) and metrizable
(since $L^1(\calX, \calC(\Act))$ is separable). If $\calX$ is countable , then we may take $\iota$ to be the counting measure. In this case,
convergence in $\D_M$ is equivalent to the pointwise convergence. See \cite[Section~2.4]{ari-bor-ghosh} for a similar discussion
in a particular settings.

For $A\in\calB(\calX)$ we define the return time to $A$ as
$$\uptau(A) \df \min\{n\geq 1\; :\; X_n\in A\}.$$
We would generally suppress the dependency of the Markov control in the notation of $\uptau$.
We shall use the notation $\Exp^v$ for the expectation with respect to the law of the controlled Markov process $\{X(v)\}$.
The following assumption ensure that every controlled Markov process is \textit{stable}.
\begin{itemize}
\item[{\bf (A3)}]
Every controlled Markov process $\{X(v)\}$, $v\in\Usm$, aperiodic, Harris recurrent and $\Lambda$-irreducible,
 for some non-negative
Borel-measure $\Lambda$ with $\Lambda(\calX)>0$. Moreover, there exists a probability measure $\nu$, $\nu(C)=1$, such that for any $A\in\calB(\calX)$, we have
\begin{equation}\label{02}
\inf_{u\in\Act}\, \inf_{x\in C} P(A | x, u)\geq \gamma \; \nu(A),
\end{equation}
for some positive constant $\gamma\in(0,1)$.
\end{itemize}
Condition \eqref{02} is generally referred to as the \emph{minorization} condition \cite{meyn-tweedie}. In view of \eqref{02}, $C$ is a
petit set and thus there exists a unique invariant probability measure for every controlled Markov process \cite[Theorem~13.01]{meyn-tweedie}. For $v\in\Usm$, we have unique invariant measure $\eta_v$
such that for all bounded Borel-measurable $f:\calX\to\R$ we have
$$\int_{\calX} f(y)\, \eta_v (dy)\; =\; \int_{\calX} \Exp_y[X_1(v)]\, \eta_v(dy).$$
Define
$$\eG\df \{\eta\in\calP(\calX)\; :\; \eta=\eta_v\; \text{for some}\; v\in\Usm\}.$$
Thus $\eG$ denote the set of all invariant probability measures. By \eqref{01} we can find $\kappa>0$ such that
\begin{equation}\label{0303}
\sup_{\mu\in\eG}\int_\calX \calV(x)\, \mu(dx)\; <\; \kappa.
\end{equation}
 \eqref{02} also implies that every stationary controlled Markov process is strongly aperiodic. Condition \eqref{02} can be seen as the counterpart of the non-degeneracy condition for controlled diffusion processes. For controlled diffusion, uniform non-degeneracy condition
is used to obtained Harnack inequality and equicontinuity property of the value functions 
(see \cite{anup-ari, ari-bor-ghosh}). In our set up, condition \eqref{02} would play a similar role, see Lemma~\ref{L6.1} below.

Now we introduce the cost function. To do so we need to specify the metric on the probability space
$\calP(\calX)$. Let $0$ be a point in $\calX$. For $p\geq 1$, we define
$$\calP_p\; \df\; \Big\{\mu\in\calP(\calX)\; :\; \int_{\calX}\, \big(\D_{\calX}(x, 0)\big)^p\, \mu(dx)<\infty\Big\}.$$
We observe that if $\D_\calX$ is a bounded metric then $\calP_p(\calX)=\calP(\calX)$ for $p\geq 1$. We shall assume that
\begin{itemize}
\item[{\bf (A4)}] for some $p\in[1, \infty)$, we have $\big(\D_\calX(\cdot, 0)\big)^p\in\sorder(\calV)$, where
$\calV$ is given by (A1).
\end{itemize}
This condition is used in various places below. One key place where this condition becomes crucial is to justify certain 
convergence of measures. Condition (A4) also allows some flexibility. For example, if we replace $\D_\calX$ by 
$\D_\calX\wedge 1$, we still retain the properties of $(\calX, \D_\calX)$ (completeness, $\sigma$-compactness, etc), but
condition (A4) holds for any $p>0$. Now we define the Wasserstein metric on $\calP_p(\calX)$
$$\eD_p(\mu_1, \mu_2) \df \inf \Big\{\int_{\calX\times\calX} (\D_\calX(y_1,y_2))^p\Theta(dy_1, dy_2)\; :\; 
\Theta\in\calP(\calX\times\calX)\; \text{has marginals}\, \mu_1, \mu_2 \Big\}^{\frac{1}{p}}.$$
It is well known that $(\calP_p(\calX), \eD_p)$ is a Polish space for $p\geq 1$. If $\D_\calX$ is bounded then $\eD_p$
gives the metric of weak convergence in $\calP(\calX)$ \cite[Theorem~7.12]{villani}. The cost function is given by the map
$$r:\calX\times\Act\times\calP_p(\calX)\to \R_+.$$
We assume that $r$ is continuous function. The following assumptions are also imposed
on $r$:
\begin{itemize}
\item[{\bf (A5)}] 
\begin{enumerate}
\item for every compact $\calK\subset\calP_p(\calX)$, $p$ is chosen from (A4), we have
$$\sup_{(u, \mu)\in\Act\times\calK} r(\cdot, u, \mu)\; \in\; \sorder(\calV)\, ;$$
\item there exists lsc functions $g_0, g_1:\calX\to\R_+$, $g_0, g_1\in\sorder(\calV)$, satisfying
\begin{equation}\label{03}
r(x, u, \frac{1}{N}\sum_{i=1}^N \del_{y_i})\; \leq \; \Big(g_0(x) + \frac{1}{N} \sum_{i=1}^N g_1(y_i)\Big),
\quad \text{for all}\, u\in\Act, \; N\geq 1,
\end{equation}
and $x, y_1, \ldots, y_N\in\calX$;
\item there exists $q\in[1, p]$ such that for any compact $K\subset\calX$, there exists positive constant $\kappa_K$ satisfying the following: for $\mu_1, \mu_2\in\calP_p(\calX)$,
\begin{equation}\label{04}
\abs{r(x, u, \mu_1)-r(x, u, \mu_2)}\; \leq \; \kappa_K\Big(1+ \int_{\calX} \big(\D_\calX(y, 0)\big)^q \mu_1(dy)
+\int_{\calX} \big(\D_\calX(y, 0)\big)^q \mu_1(dy)\Big)^{q'}\, \eD_q(\mu_1, \mu_2),
\end{equation}
for all $x\in K$ and $u\in\Act$ where $q'=\frac{q-1}{q}$.  In addition, for any $(x, u)\in\calX\times\Act$ and any compact set $K$ there exists constant $\theta$, may depend
 on $x, u, K$, such that 
 \begin{equation}\label{extra}
 \abs{r(x, u, \frac{1}{N}\sum \del_{y_i})- r(x, u, \frac{1}{N-1}\sum \del_{y_i})}\leq \frac{\theta}{N}, \quad y_i\in K, \; i=1, \ldots, N.
\end{equation}
\end{enumerate}
\end{itemize}
In Example~\ref{E-1} and Example~\ref{E-2},  we show that assumptions in (A5) are satisfied by a large class of cost functions. To prove existence of a mean field game
solutions we do not need all the assumptions in (A5). All these conditions are used to justify the convergence
of equilibrium from $N$-person game to a MFG solution.
We also assume that
\begin{itemize}
\item[{\bf (A6)}] if $v_n\to v$ as $n\to \infty$ in $(\Usm, \D_M)$, then $\norm{\eta_{v^n}-\eta_{v}}_{TV}\to 0$ as $n\to\infty$ where $\eta_{v_n},
\eta_v$ are the invariant measures corresponding to the stationary Markov controls $v_n, v,$ respectively.
\end{itemize}
In view of (A1), we note that $\eG$ is tight and therefore $\eta_{v^n}$ weakly converges to $\eta_{v}$. If $\calX$ is countable (discrete), this implies that
$\norm{\eta_{v^n}-\eta_{v}}_{TV}\to 0$ as $n\to\infty$.

Let us now give some examples of cost functions that satisfy (A5).
\begin{example}\label{E-1}
Let $r(x, u, \mu) = R(x, u, \zeta(x, \mu))$ where $R:\calX\times\Act\times\R$ is a continuous function and for every compact $K\subset\calX$ there exists 
constant $\gamma_K$ satisfying
\begin{equation}\label{Eg1}
\abs{R(x, u, z)-R(x, u, z_1)}\; \leq \; \gamma_K\,  \abs{z-z_1}, \quad \text{for}\;  \; z, \, z_1\in\R, \quad (x, u)\in 
K\times\Act.
\end{equation}
Also suppose that for some $g_0\in\sorder(\calV)$ we have
$$R(x, u, z) \leq g_0(x) + \kappa \abs{z}, \quad \forall \; x\in\calX, \;  z\in\R,$$
for some constant $\kappa>0$.
Let $m\in\NN$ be fixed, and $\zeta(x, \mu)$ is given as follows
$$\zeta(x, \mu)\;\df\; \underbrace{\int_{\calX}\cdots\int_{\calX}}_{m\, \text{times}} \varphi(\D_\calX(x, y_j))\, \Pi_{j=1}^m \mu(dy_j),$$
where $\varphi:\R\to\R$ is bounded and Lipschitz continuous. Thus (A5)(2) is satisfied by $r$.To see \eqref{04} holds, we consider 
$\mu, \tilde{\mu}\in\calP(\calX)$. We let $m=2$ for simplicity. Due to \eqref{Eg1}, we only need to compute the following.
Let $\Theta\in\calP(\calX\times\calX)$ be such that $\mu, \tilde{\mu}$ are its marginals. Then
\begin{align*}
\zeta(x, \mu_1)- \zeta(x, \mu) &= \int_{\calX\times\calX} \Pi_{j=1}^2\varphi(\D_{\calX}(x, y_j))\,  \mu(dy_1)\mu(dy_2)
-\int_{\calX\times\calX} \Pi_{j=1}^2\varphi(\D_\calX(x, y_j))\,  \tilde{\mu}(dy_1)\tilde{\mu}(dy_2)
\\
&= \int_{\calX} \varphi(\D_\calX(x, y_1))\int_{\calX^2}(\varphi(\D_\calX(x, y_2))-\varphi(\D_\calX(x, y_3))) \Theta(dy_2, dy_3) \mu(dy_1)
\\
&\, \quad + \int_{\calX} \varphi(\D_\calX(x, y_3))\int_{\calX^2}(\varphi(\D_\calX(x, y_1))-\varphi(\D_\calX(x, y_2))) \Theta(dy_1, dy_2) \tilde{\mu}(dy_3)
\\
&\leq 2\,\abs{\varphi}_\infty\, \abs{\varphi}_{Lip} \int_{\calX^2} \D_\calX(y_2, y_3)\, \Theta(dy_2, dy_3),
\end{align*}
where $\abs{\varphi}_\infty,\, \abs{\varphi}_{Lip}$ denote the supremum norm of $\varphi$ and respectively, the
Lipschitz constant of $\varphi$.
$\Theta$ being arbitrary we have \eqref{04} with $q=1$. \eqref{extra} is easy to check.
\end{example}

\begin{example}\label{E-2}
Consider $r$ and $R$ as defined in Example~\ref{E-1}. Define for continuous $\varphi:\calX\times\calX\to R_+$,
$$\zeta(x, \mu)\;\df\; \int_{\calX}\varphi(x, y) \, \mu(dy),$$
where for some $q\in[1, p]$ and $x\in K$, $K\subset\calX$ compact, we have
\begin{equation}\label{Eg2}
\abs{\varphi(x, y)-\varphi(x, y_1)} \leq \kappa_K\, (1+\D^{q-1}_\calX(0, y) + \D^{q-1}_\calX(0, y_1))\, \D_\calX(y, y_1),
\end{equation}
for $x\in K$. Also assume that for some $\tilde{g}_0\in\sorder(\calV)$ we have 
\begin{equation}\label{Eg3}
\abs{\varphi(x, y)} \leq \tilde{g}_0(x) + \tilde{g}_0(y), \quad x, y\in\calX.
\end{equation}
From \eqref{Eg3} we have $\varphi(x, \cdot)\in\sorder(\calV)$ uniformly for $x$ in some compact subset of $\calX$.
We see from \eqref{Eg3} that (A2)(2) holds. To see that (A5)(1) holds we note that for any compact set 
$\calK\subset\calX$, we have $\sup_{\mu\in\calK}\int (\D_\calX(x, 0))^p \mu(dx)<\infty$ by 
\cite[Theorem~7.12]{villani}. Therefore using (A4) we have $\sup_{\mu\in\calK}\int \calV(x) \mu(dx)<\infty$ which
combined with \eqref{Eg3} gives (A5)(1).
Now let $\Theta$ be an element of $\calP(\calX\times\calX)$ with marginals $\mu, \tilde{\mu}$. Then for $x\in K$,
using \eqref{Eg2} we have
\begin{align*}
\zeta(x, \mu)-\zeta(x, \tilde{\mu}) & = \int_{\calX\times\calX} \Big(\varphi(x, y)-\varphi(x, y_1)\Big) \, \Theta(dy, dy_1)
\\ 
&\leq \kappa_K\, \int_{\calX\times\calX}  (1+\D^{q-1}_\calX(0, y) + \D^{q-1}_\calX(0, y_1))\, \D_\calX(y, y_1) \, \Theta(dy, dy_1)
\\
&\leq \kappa_1 \Big[\int_{\calX\times\calX}  (1+\D^{q}_\calX(0, y) + \D^{q}_\calX(0, y_1))\, \Theta(dy, dy_1)\Big]^{\frac{q-1}{q}} 
\\
&\, \qquad \Big[\int_{\calX\times\calX} (\D_\calX(y, y_1))^q \Theta(dy, dy_1)\Big]^{\frac{1}{q}},
\end{align*}
where in the last line we use H\"{o}lder inequality. $\Theta$ begin arbitrary we get from above that
\begin{align*}
\abs{\zeta(x, \mu)-\zeta(x, \tilde{\mu})}\leq \kappa_1 \Big[(1+\int_{\calX}\D^{q}_\calX(0, y)\mu(dy) + 
\int_{\calX} \D^{q}_\calX(0, y)\, \tilde{\mu}(dy))\Big]^{q'} \eD_q(\mu, \tilde{\mu}).
\end{align*} 
Thus \eqref{04} is satisfied. One can also check that \eqref{extra} holds.
\end{example}

Before we conclude the section let us introduce the Lemma~\ref{L5.1}.
For compact set $K\subset\calX$ containing $0$ we define the following map,
\[\mathfrak{P}(x)=\mathfrak{P}_K(x)\; \df\;
\left\{\begin{array}{lll}
x & \text{if}\; x\in K,
\\[2mm]
0 & \text{otherwise}.
\end{array}
\right.
\]
For $\mu\in\calP(\calX)$, we define $\tilde{\mu}_K(B)=\mu(\mathfrak{P}_K^{-1}(B))$ for all $B\in\calB(\calX)$.
Then the following result follows by mimicking the arguments in \cite[Lemma~4.1]{anup-ari} and using (A4), (A5).
\begin{lemma}\label{L5.1}
Let $\eps>0$ be given. Then for any compact $K\subset\calX$, there exists compact $\tilde{K}=\tilde{K}(K)\subset\calX$
 such that if $(\mu_1, \ldots, \mu_N)\in\eG^N$ then
$$\sup_{(x, u)\in K\times\Act}\Big|\int_{(\calX)^N}r(x, u, \frac{1}{N}\sum_{j=1}^N \delta_{y_j}) \Pi_{j=1}^N \mu_j(dy_j) -
\int_{(\calX)^N}r(x, u, \frac{1}{N}\sum_{j=1}^N \delta_{y_j}) \Pi_{j=1}^N \tilde\mu_j(dy_j)\Big| \leq \eps,$$
for all $N\geq 1$ where $\tilde{\mu}_j= (\tilde{\mu}_j)_{\tilde{K}}$.
\end{lemma}
\begin{proof}
We note that for $x_j, y_j\in\calX, 1\leq j\leq N$, we have for any $p\geq 1$,
\begin{equation}\label{E1}
\textstyle\eD_p\Bigl(\frac{1}{N}\sum_{j=1}^N\delta_{x_j},
\frac{1}{N}\sum_{j=1}^N\delta_{y_j}\Bigr)\;\leq\;
\Bigl(\frac{1}{N}\sum_{i=1}^N (\D_\calX(x_j,y_j))^p\Bigr)^{\nicefrac{1}{p}}.
\end{equation}
For any compact $\tilde{K}$ we denote $\tilde{\mu}_j= (\tilde{\mu}_j)_{\tilde{K}}$. Then for all $(x, u)\in K\times\Act$, we have
\begin{align*}
& \Big|\int_{\calX^{N}} r(x, u, \frac{1}{N}\sum_{j=1}^N\delta_{y_j})\Pi_{j=1}^n \mu_j(d{y_j})
- \int_{\calX^{N}} r(x, u, \frac{1}{N}\sum_{j=1}^N\delta_{y_j})\Pi_{j=1}^n \tilde{\mu}_j(d{y_j})\Big|
\\
&=\; \Big|\int_{\calX^{N}} \Bigl(r(x, u, \frac{1}{N}\sum_{j=1}^N\delta_{y_j})
-  r(x, u, \frac{1}{N}\sum_{j=1}^N\delta_{\mathfrak{P}_{\tilde{K}}(y_j)})\Bigr)\Pi_{j=1}^n \mu_j(d{y_j})\Big|
\\
&\leq\; \kappa_K \Big|\int_{\calX^{N}} \Bigl(\Big(1+ \frac{1}{N}\sum_{j=1}^N (\D_\calX(y_j, 0))^q 
+ \frac{1}{N}\sum_{j=1}^N (\D_\calX(\mathfrak{P}_{\tilde{K}}(y_j), 0)^q\Big)^{q'}
\\
&\, \hspace{2in} \eD_q\big(\frac{1}{N}\sum_{j=1}^N\delta_{y_j}, \frac{1}{N}\sum_{j=1}^N\delta_{\mathfrak{P}_{\tilde{K}}(y_j)}\big) \Bigr)\Pi_{j=1}^n \mu_j(d{y_j})\Big|
\\
&\leq\; \kappa_K \Big|\int_{\calX^{N}} \big(1+ \frac{1}{N}\sum_{j=1}^N (\D_\calX(y_j, 0))^q + \frac{1}{N}\sum_{j=1}^N (\D_\calX(\mathfrak{P}_{\tilde{K}}(y_j), 0)^q) \big)\Pi_{j=1}^N \mu_j(d{y_j})\Big|^{q'}
\\
&\, \qquad \times\; \Big|\int_{\calX^{N}} \Bigl[\eD_q\big(\frac{1}{N}\sum_{j=1}^N\delta_{y_j}, \frac{1}{N}\sum_{j=1}^N\delta_{\mathfrak{P}_{\tilde{K}}(y_j)}\big)\Bigr]^q \Pi_{j=1}^N \mu_j(d{y_j})
\Big|^{1/q}
\\
&\leq\; \kappa_2 \Big|\int_{\calX^{N}} \frac{1}{N}\sum_{j=1}^N \D_\calX(y_j,\mathfrak{P}_{\tilde{K}}(y_j))^q\, \Pi_{j=1}^N\mu_j(d{y_j})\Big|^{1/q}
\\
&=\; \kappa_2 \Big| \frac{1}{N}\sum_{i=1}^N \int_{\tilde{K}^c}\D_\calX(y_j, 0)^q\,  \mu_j(d{y_j})\Big|^{1/q}
\end{align*}
where in the third inequality we use (A4), \eqref{0303}, \eqref{E1}. Now we use (A4) and choose $\tilde{K}$ to be large enough to satisfy
$$\int_{\tilde{K}^c} (\D_{\calX}(y_j, 0))^q \mu_j(dy_j)\leq 
\Big[\int_{\tilde{K}^c} (\D_{\calX}(y_j, 0))^p \mu_j(dy_j)\Big]^{\frac{q}{p}}\leq \eps.$$
Thus combining above to display we have the result.
\end{proof}
In view of Lemma~\ref{L5.1}, we obtain that for $\mu_1, \ldots, \mu_N\in\eG$ the following map
\begin{equation}\label{007}
(x, u)\mapsto \int_{\calX\times\cdots\times\calX} r(x, u, \frac{1}{N}\sum_{i=1}^N \delta_{y^i})\, \Pi_{i=1}^N \mu_i(dy^i),\end{equation}
is continuous.

\section{\bf Mean Field Games}\label{S-MFG}
In this section we introduce the mean field games (MFG) for the limiting dynamics of $N$-person games. The ergodic cost of the player is given 
by 
\begin{equation}\label{cost}
J(x, v, \mu)\; \df\; \limsup_{n\to\infty} \frac{1}{n}\; \Exp^v_x\Big[\sum_{i=0}^{n-1} r(X_i, v_i(X_i), \mu)\Big].
\end{equation}
Here $v\in\Uadm$ and $\mu\in\calP_p(\calX)$. The goal of the player is to optimize the cost \eqref{cost} given $\mu$. Define
\begin{equation}\label{05}
\varrho_\mu\; \df\; \inf_{v\in\Uadm} \, J(x, v, \mu).
\end{equation}
Note that $\varrho_\mu$ does not depend on $x$. The cost being ergodic, it is intuitive that $J(x, v, \mu)$ should depend on the \emph{equilibrium}
distribution of the corresponding process $\{X(v)\}$. For every $\mu\in\calP_p(\calX)$, we define
$$r_\mu\colon \calX\times\Act\to \R_+, \quad \text{as}\quad r_\mu(x, u)= r(x, u, \mu).$$
Our first result is the following
\begin{theorem}\label{Thm-1}
Let (A1)--(A3), (A5)(1) hold.
There exists a unique pair $(V_\mu, \varrho)\in\calC(\calX)\times\R, \, V(0)=0,\, V\in\sorder(\calV)$ that satisfies
\begin{equation}\label{06}
V_\mu(x) + \varrho \; =\; \min_{u\in\Act}\Big\{r_\mu(x, u) + \int_{\calX} V_\mu(y)\, P(dy|x, u)\Big\}.
\end{equation}
Moreover, we have $\varrho=\varrho_\mu$. The $v\in\Usm$ is optimal for \eqref{05} if and only if it is a minimizing selector
of \eqref{06}.
\end{theorem} 
Theorem~\ref{Thm-1} is not something new. Ergodic control problems for controlled Markov chain have been
 studied vastly in literature (see for example, \cite{ari-et-al, lerma-lasserre} and the references therein). In section~\ref{Ergodic HJB} we obtain a representation formula
for $V_\mu$ which plays crucial role in this article. Analogous representation formula under the
settings of controlled diffusion can found in \cite{ari-bor-ghosh}.
We use the method of split chain to show that $\calV_\mu$ is in $\calC(\calX)$ and forms a family of equi-continuous functions if
$\mu$ lies in a compact subset of $\calP_p(\calX)$.

Thus by Theorem~\ref{Thm-1} we have unique $V_\mu$ for every $\mu\in\calP(\calX)$. Since $V_\mu\in \sorder(\calV)$ and
$\int_{\calX} \calV(y)\, P(dy|x, u)$ is finite (by \eqref{01}) we see that
$$u\mapsto \int_{\calX} V_\mu(dy) P(dy|x, u), \quad \textit{is continuous},$$ 
for every $x$. Hence there exists a measurable selector of \eqref{06} by 
\cite[Proposition~7.33]{bertsekas-shreve}. 
A control $v\in\Usm$ is said to be a minimizing selector of \eqref{06} if
\begin{equation}\label{06.5}
V_\mu(x) + \varrho_\mu = \int_{\Act}r_\mu(x, u)v(du|x) + \int_{\calX\times\Act} V_\mu(y) P(dy|x, u)v(du|x),\quad \text{almost surely }\; \eta_v\,. 
\end{equation}
Since any measurable selector of \eqref{06} satisfies \eqref{06.5}, the set of minimizing selector is non-empty.
We define
\begin{equation}\label{07}
\calA(\mu)\df \{\eta\in\eG\, :\, \text{where}\; \eta=\eta_v\; \text{for some minimizing selector }\, v\; \text{of}\; \eqref{06}\}.
\end{equation}
Thus $\mu\mapsto \calA(\mu)$ is set valued.
\begin{definition}\label{defi-1}
A probability measure $\eta\in\eG$ is said to be a MFG solution if $\eta\in\calA(\eta)$ and
 therefore we have a unique $V_\eta\in\calC(\calX)\cap\sorder(\calV)$, depending on $\eta$,
and $v\in\Usm$ satisfying
\begin{align}
\min_{u\in\Act}\Big\{r_\eta(x, u) + \int_{\calX} V_\eta(y)\, P(dy|x, u)\Big\} 
& = r_\eta(x, u) + \int_{\calX} V_\eta(y)\, P(dy|x, u) \nonumber
\\
&=\varrho_\eta + V_\eta(x)
, \;\; \text{almost surely}\; \eta,\label{08}
\end{align}
and,
\begin{align}
\int_{\calX} f(dy) \, \eta(dy) = \int_{\calX} \Exp_y[X_1(v)]\, \eta(dy), \quad \text{for all}\, f\in \calC_b(\calX). \label{09}
\end{align}
\end{definition}
Thus existence of MFG solution is related to the existence of fixed point to the map $\calA$. 
\eqref{08} represents the HJB corresponding to the ergodic  control problem with running cost $r_\eta$ and \eqref{09} characterizes $\eta$ as the invariant measure corresponding to $v$ where $v$ is a minimizing selector of
\eqref{08}.

We show that under above assumptions there exists a fixed point.
\begin{theorem}\label{Thm-2}
Assume that (A1)--(A4), (A5)(1), (A6) holds. Then there exists a MFG solution in the sense of 
Definition~\ref{defi-1}.
\end{theorem}
Proof of Theorem~\ref{Thm-2} is similar to the proof of Theorem~\ref{Thm-3}. Therefore we defer the proof
until Section~\ref{S-Nash}. 

\begin{remark}[Uniqueness]\label{R-uni}
In this article we are interested in existence of MFG solution and its relation with the equilibrium of $N$-person games. However, one can impose
$L^2$ type monotonicity condition (similar to \cite{lasry-lions}) on $r$ which would give a unique MFG solution in the sense of Definition~\ref{defi-1}. 
For $v\in\Usm$, let us denote $\eta\cast v (dx, du)= v(du|x)\eta_v(dx)$. Now if we impose the following monotonicity assumption on $r$,
$$\int_{\calX\times\Act} \Big(r(x, u, \eta)-r(x, u, \bar{\eta})\Big)\Big(\eta\cast v(dx, du)-\bar{\eta}\cast \bar{v}(dx, du)\Big)\leq 0, 
\quad \text{implies}\;\; \eta=\bar{\eta},$$
then one can easily establish the uniqueness of MFG solution (see for example, \cite{cardaliaguet, lasry-lions}). In fact, if 
the cost is of the form $r(x, u, \mu)\equiv r_1(x, u) + r_2(x, \mu)$, then the above monotonicity condition is similar to the one appeared in 
\cite{cardaliaguet, lasry-lions}.
\end{remark}

\section{\bf N-person Games and Nash Equilibria}\label{S-Nash}
In this section we describe the $N$-person game and show that under above hypothesis there exists an equilibrium for the $N$-person game.
These equilibriums are the central object to the mean field games. It is generally expected that the MFG solution could be used to approximate the
$N$-person game equilibriums. In this article we show that as $N\to\infty$, the equilibrium solutions for the $N$-person game tend to the solutions of MFG.
$N$-person game is described as follows: there are $N$-independent players, each of them is choosing his/her controls from
$\Uadm$. The choice of control is made independent from the other players. Therefore the players does not 
have access to the \emph{full information}.
 The players interact with each other through their running cost.
To introduce the cost we define for $\mathbf{x}=(x_1, \ldots, x_N)$,
$$\breve{\mathbf{x}}_i = (x_1, \ldots, x_{i-1}, x_{i+1}, \ldots, x_N).$$
The running cost for the $i$-th player is given by
$$r^N_i(x_i, u, \breve{\mathbf{x}}_i)= r(x_i, u, \frac{1}{N-1}\sum_{j\neq i}\del_{x_j}).$$
Hence the cost of every player is influenced by the empirical distribution of other players. Therefore interaction
with the other players is \emph{weak}.
Denote
\begin{align*}
\eG^N = \eG\times\cdots\times\eG.
\end{align*}
For $\boldsymbol{\uppi} =(\uppi_1, \ldots, \uppi_N)\in\eG^N$, we define
\begin{equation}\label{9.5}
\Hat{r}^N_i (x, u, \bm{\uppi})\; \df \; \int_{\calX\times\cdots\times\calX} r(x, u, \frac{1}{N-1}\sum_{j\neq i}\del_{y_j})\, \Pi_{j\neq i} \uppi_{j}(dy_j).
\end{equation}
By \eqref{007}, $(x, u)\mapsto \Hat{r}^N_i(x, u) $ is continuous. 
For $\boldsymbol{\uppi}\in\eG^N$ and $v\in\Uadm$, we define the cost of the $i$-th player as
\begin{equation}\label{9.55}
J^N_i(x, v, \bm{\uppi})\; \df\; \limsup_{n\to\infty}\, \frac{1}{n}\, \Exp_x^v\Big[ \sum_{j=0}^{n-1}\Hat{r}^N_i(X_j, v_j, \bm{\uppi})\Big].
\end{equation}
\begin{definition}\label{defi-2}
$\bm{\uppi}\in\eG^N$ is said be a Nash equilibrium for the $N$-person game if for every $i\in{1, \ldots, N}$ we have
$v(i)\in\Usm$ satisfying the following
\begin{equation}\label{10}
J^N_i(x, v, \bm{\uppi})\; \geq\; J^N_i(x, v(i), \bm{\uppi}), \quad \text{for all}\, x\in\calX, \; v\in\Uadm,
\end{equation}
and if $\bm{\uppi} = (\uppi_1, \ldots, \uppi_N)$ then
\begin{equation}\label{11}
\uppi_i = \eta_{v(i)}, \quad i.e., \; \int_{\calX} f(y)\, \uppi_i(dy)\; =\; \int_{\calX} \Exp_x^{v(i)}[f(X_1)]\, \uppi_i(dy),
\end{equation}
for all bounded measurable function $f$.
\end{definition}
By the above definition, existence of Nash equilibrium would imply existence of an \emph{equilibrium environment} 
for every player and no player has anything to gain by changing only their own equilibrium strategy. The optimality 
criterion of the equilibrium is understood through \eqref{10} and \eqref{11} shows that the equilibrium strategy 
leads to the invariant distribution.

\begin{remark}
The definition of Nash equilibrium in Definition~\ref{defi-2} weaker than the one used in \cite{anup-ari, feleqi}.
However one can impose mild additional conditions to make them equivalent.
A Nash equilibrium is generally defined to be a collection of (stationary Markov) strategies $(v(1), \ldots, v(N))$ such that
for all $i\in\{1, \ldots, N\}$ one has
\begin{align*}
&\limsup_{n\to\infty}\, \frac{1}{n}\, 
\Exp\Big[ \sum_{j=0}^{n-1} r^N_i\Big(X^{(i)}_j(v), v_j, \frac{1}{N-1}\sum_{k\neq i}\del_{X^{(k)}_j(v(k))}\Big)\Big]
\\ 
&\geq\; 
\limsup_{n\to\infty}\, \frac{1}{n}\, 
\Exp\Big[ \sum_{j=0}^{n-1} r^N_i\Big(X^{(i)}_j(v(i)), v_j(i), \frac{1}{N-1}\sum_{k\neq i}\del_{X^{(k)}_j(v(k))}\Big)\Big].
\end{align*}
Here $\{X^{(k)}\}$ denotes the controlled process for the $k$-th player. 
Since the players are independent, it is not hard to show that for some $\tilde{\bm{\uppi}}$ and $\tilde{v}\in\Usm$
\begin{align*}
\limsup_{n\to\infty}\, \frac{1}{n}\, 
\Exp\Big[ \sum_{j=0}^{n-1}\Hat{r}^N_i\Big(X^{(i)}_j(v), v_j, \frac{1}{N-1}\sum_{k\neq i}\del_{X^{(k)}_j(v(k))}\Big)\Big]
\geq \int_{\calX} \hat{r}^N_i(y, \tilde{v}(y), \tilde{\bm{\uppi}})\, \tilde\uppi_i(dy),
\end{align*}
and $\tilde\uppi_i$ is the invariant measure corresponding to the control $\tilde{v}$. Using (A1) and (A5) it is also easy to show that 
$$\int_{\calX} \hat{r}^N_i(y, \tilde{v}(y), \tilde{\bm{\uppi}})\, \tilde\uppi_i(dy)= J_i^N(x, \tilde{v}, \tilde{\bm{\uppi}}), 
\quad \text{for all}\; x\in\calX.$$
This also shows that it is enough to minimize over the controls in $\Usm$.
One the other hand, if the equilibrium strategies $v(i)$ is in $\Usm$ for all $i$, then the joint Markov process
$\Phi\df\{X^{(1)}(v(1)), \ldots, X^{(N)}(v(N))\}$ has a unique invariant measure.
To show the existence of unique
invariant measure for $\Phi$ we observe from \eqref{01} that there is 
a stability inequality for this process and $C\times\cdots\times C$ is a small set with respect
to the measure $\nu\times\cdots\times\nu$ by \eqref{02}. Since product of invariant measures of $\{X^{(i)}(v(i)\}$ is an invariant measure we obtain that
the Markov chain $\Phi$ recurrent \cite[Proposition~10.1.1]{meyn-tweedie}. Hence uniqueness
of invariant measure follows from \cite[Theorem~10.0.1]{meyn-tweedie}.

Now if we assume that $\text{supp}(\nu)$ has non-empty interior, then $\Phi$ being  irreducible and aperiodic
(with respect to $\nu\times\cdots\times\nu$) Markov chain we see that every compact
subset of $(\calX)^N$ is a petit set by \cite[Proposition~6.2.8]{meyn-tweedie}.
Therefore using (A1) and \cite[Theorem~14.0.1]{meyn-tweedie} we get
$$
\limsup_{n\to\infty}\, \frac{1}{n}\, 
\Exp\Big[ \sum_{j=0}^{n-1} r^N_i\Big(X^{(i)}_j(v(i)), v_j(i), \frac{1}{N-1}\sum_{k\neq i}\del_{X^{(k)}_j(v(k))}\Big)\Big]
=J_i^N(x, v(i), \bm{\uppi})$$
where $\bm\uppi$ satisfies \eqref{11}. \end{remark}

Our main result of this section is as follows
\begin{theorem}\label{Thm-3}
Assume that (A1)--(A6) hold. Then there exists a Nash equilibrium in the sense of Definition~\ref{defi-2} for the $N$-person game.
\end{theorem}

The proof technique of Theorem~\ref{Thm-3} uses fixed point argument. Use of fixed point argument is quite standard in mean field game theory and
it has been used in \emph{almost all} the existing works to obtain the existence of Nash equilibrium or MFG solutions.
Our arguments are similar to the one used in \cite{anup-ari}.

We recall the set $\eG$ of invariant measures.
$$\eG\df \{\eta\in\calP(\calX)\; :\; \eta=\eta_v\; \text{for some}\; v\in\Usm\}.$$
Our next result shows that $\eG$ is compact and convex.
\begin{lemma}\label{L4.1}
The set $\eG$ defined above is a convex set. It is also compact in the topology of weak convergence.
\end{lemma}
\begin{proof}
From \eqref{0303} we see that $\eG$ is tight in $\calP(\calX)$. Since $(\Usm, \D_M)$ is compact, using (A2)
and (A6) we have
$\eG$ compact.
To show the convexity we consider $\eta_{v^i}\in\eG$ for $i=1,2$
where $v^i\in\Usm$ for $i=1, 2$. Let $\bar\eta = \theta \eta_{v^1} + (1-\theta) \eta_{v^2}$ where $\theta\in[0,1]$. We want to show that $\bar\eta=\eta_v$ for some $v\in\Usm$.
Define
$$v(du|x)\df \theta\frac{d\eta_{v^1}}{d\bar\eta}(x) v^1(du|x) + (1-\theta)\frac{d\eta_{v^2}}{d\bar\eta}(x) v^2(du|x), \quad x\in\calX.$$
Here $\frac{d\eta_{v^i}}{d\bar\eta}(x)$ denotes the Radon-Nykodim derivative. The above definition makes sense as without loss of generality we can choose
$$\theta\frac{d\eta_{v^1}}{d\bar\eta}(x) + (1-\theta)\frac{d\eta_{v^2}}{d\bar\eta}(x)=1, \quad \text{for all}\; x\in\calX.$$
To complete the proof we need to show that $\bar\eta=\eta_v$ where $v$ is given above. To show this we consider $f\in\calC_b(\calX)$. Then
\begin{align*}
\int_{\calX} \Exp^v_y[f(X_1)] \, d\bar\eta & = \int_\calX\Big[\theta \frac{d\eta_{v^1}}{d\bar\eta}(y)\int_{\calX\times\Act} f(z) P(dz|y, u) v^1(du|y)
\\
&\, \qquad +(1-\theta) \frac{d\eta_{v^2}}{d\bar\eta}(y)\int_{\calX\times\Act} f(z) P(dz|y, u) v^2(du|y)\Big]\bar\eta(dy)
\\
&= \int_\calX \theta \int_{\calX\times\Act} f(z) P(dz|y, u) v^1(du|y)\, \eta_{v^1}(dy)
\\
&\, \qquad +(1-\theta) \frac{d\eta_{v^2}}{d\bar\eta}(y)\int_{\calX\times\Act} f(z) P(dz|y, u) v^2(du|y)\, \eta_{v^2}(dy)
\\
&=\int_{\calX} f(y)\, \bar\eta(dy).
\end{align*}
By uniqueness of invariant measure we obtain $\bar\eta=\eta_{v}$. Hence the proof.
\end{proof}

Any element of $\eG^N$ will be denoted in boldface fonts/symbols. For exmaple, $\bm{\uppi}$ is an element in $\eG^N$ whereas 
$\uppi$ would denote an element of $\eG$. For any $\bm{\uppi}$, recall $\Hat{r}^N_i$ from $\eqref{9.5}$. 
Using \eqref{007} and Theorem~\ref{Thm-1}
we have for each $i\in\{1, \ldots, N\}$, a unique pair $(V^N_{i,\bm{\uppi}}, \varrho^N_{i, \bm\uppi})\in\calC(\calX)\cap\sorder(\calV)\times \R$, 
$V^N_{i,\bm{\uppi}}(0)=0$, satisfing
\begin{equation}\label{12}
V^N_{i,\bm{\uppi}}(x) + \varrho^N_{i, \bm\uppi} = \min_{u\in\Act}\{\Hat{r}^N_i(x, u, \bm{\uppi}) + \int_{\calX} V^N_{i,\bm{\uppi}}(y) P(dy|x, u) \},
\end{equation}
where $\varrho^N_{i, \bm\uppi}$ is the optimal value of $J^N_i$. Also note that all the stationary Markov controls are stable due to (A1). By a 
minimizing selector of \eqref{12} we mean control $v\in\Usm$ satisfying
\begin{equation}\label{13}
V^N_{i,\bm{\uppi}}(x) + \varrho^N_{i, \bm\uppi} = \int_{\Act}\Hat{r}^N_i(x, u, \bm{\uppi})v(du|x) + \int_{\calX\times\Act} V^N_{i,\bm{\uppi}}(y) P(dy|x, u)v(du|x),
\quad \text{almost surely }\; \eta_v\,. 
\end{equation}
We note that any deterministic measurable selector of \eqref{12} satisfies \eqref{13}. Thus the set of minimizing selector in the sense of \eqref{13} is non-empty.
Given $\bm{\uppi}\in\eG^N$, we define $\bm{\calA}(\bm{\uppi})=\calA_1\times\cdots\times\calA_N\subset\eG^N$ as follows
$$\calA_i(\bm{\uppi})\;\df\; \{\eta\in\eG \; :\; \eta=\eta_{v^i}\; \text{where}\; v^i \; \text{satisfies}\ \eqref{13}\}.$$
It is easy to see that existence of fixed point of the map $\bm{\uppi}\mapsto \bm{\calA}(\bm{\uppi})$ would establish Theorem~\ref{Thm-3}. 
We also observe that the definition of $\bm{\calA}$ is similar to the map $\calA$ defined in \eqref{07}.  The next result shows that
$ \bm{\calA}(\bm{\uppi})$ is compact and convex for every $\bm{\uppi}$.

\begin{lemma}\label{L4.2}
The map $\bm{\uppi}\ni\eG^N\mapsto \bm{\calA}(\bm{\uppi})\, (\neq \emptyset)$
 is compact (with respect to weak topology) and convex set valued.
\end{lemma}

\begin{proof}
Form Theorem~\ref{Thm-1} we have
$$\varrho^N_{i, \bm{\uppi}} = \int_{\calX\times\Act} \Hat{r}^N_i(y, u, \bm{\uppi})\tilde{v}(du|y)\, \eta_{\tilde{v}}(dy),$$
where $\tilde{v}$ is a minimizing selector in the sense of \eqref{13}. Again for $\eta_{v^i}$ in $\calA_i(\bm{\uppi})$ and
$\theta\in[0,1]$, we can have $v$ satisifying
(see the construction in Lemma~\ref{L4.1})
$$\eta_v(dy) v(du|y) = \theta\, \eta_{v^1}(dy) v^1(du|y) + (1-\theta)\,\eta_{v^2}(dy) v^2(du|y),
\quad \text{and}\quad \eta_{v} = \theta\eta_{v^1} + (1-\theta) \eta_{v^2}.$$
Therefore we have 
\begin{equation}\label{13.5}
\varrho^N_{i, \bm{\uppi}} = \int_{\calX\times\Act} \Hat{r}^N_i(y, u, \bm{\uppi})v(du|y)\, \eta_v(dy),
\end{equation}
and thus $v$ is an optimal control. Again using Theorem~\ref{Thm-1} we see that $v$ is a minimizing selector in the sense of \eqref{13}. Thus
$\eta_v\in\calA_i(\bm{\uppi})$. This shows that $\calA_i(\bm{\uppi})$ is convex and hence $\bm{\calA}(\bm{\uppi})$ is convex.
To establish compactness it is enough to show that each $\calA_i(\bm{\uppi})$ is compact. 
$(\Usm, \D_M)$ and $\eG$ being compact we only need to show that if $v_n\to v$ then 
\begin{equation}\label{14}
\int_{\calX\times\Act} f(y, u) v^n(du|y)\eta_{v^n}(dy)\to \int_{\calX\times\Act} f(y, u) v(du|y)\eta_{v}(dy),
\end{equation}
for all $f\in\calC_b(\calX\times\Act)$. Since by (A5)(2) we have $\max_{u\in\Act}\Hat{r}^N_i(\cdot, u, \bm{\uppi})
\in\sorder(\calV)$, using \eqref{14} we get \eqref{13.5} which implies that $\eta_v\in\calA_i(\bm{\uppi})$.
 To show \eqref{14} we use (A2) and (A6). 
\begin{align*}
&\Big|\int_{\calX\times\Act} f(y, u) v^n(du|y)\eta_{v^n}(dy)-\int_{\calX\times\Act} f(y, u) v(du|y)\eta_{v}(dy)\Big|
\\
&= \Big|\int_{\calX} \Exp^{v^n}_y[f(X_1, v^n(X_1)] \eta_{v^n}(dy)-\int_{\calX} \Exp^{v}_y[f(X_1, v(X_1)] \eta_{v}(dy)\Big|
\\
&\leq \norm{f}_\infty \norm{\eta_{v^n}-\eta_{v}}_{TV} + 
\Big|\int_{\calX} \Exp^{v^n}_y[f(X_1, v^n(X_1)] \eta_{v}(dy)-\int_{\calX} \Exp^{v}_y[f(X_1, v(X_1)] \eta_{v}(dy)\Big|
\\
&\to 0,
\end{align*}
as $n\to\infty$, where in the third line we use (A6) and (A2) with dominated convergence theorem.
This shows that $\calA_i(\bm{\uppi})$ is compact. Hence the proof.
\end{proof}

\begin{definition}
The map $\bm{\uppi}\mapsto \bm{\calA}(\bm{\uppi})$ is said to be \emph{upper-hemicontinuous} if for $\bm{\uppi}_n\to\bm{\uppi}$
as $n\to\infty$, and $\tilde{\bm\uppi}_n\in\bm{\calA}(\bm{\uppi})$ for all $n$, then the sequence $\{\tilde{\bm{\uppi}}_n\}$ has a limit point in
$\bm{\calA}(\bm{\uppi})$.
\end{definition}

Upper-hemicontinuity plays a key role to obtain the fixed point of $\bm{\uppi}\mapsto \bm{\calA}(\bm{\uppi})$. The next result shows that the 
map is in fact, upper-hemicontinuous.
\begin{lemma}\label{L4.3}
The map $\bm{\uppi}\mapsto \bm{\calA}(\bm{\uppi})$ is upper-hemicontinuous.
\end{lemma}

\begin{proof}
Let $\bm{\uppi}_n\to\bm{\uppi}$ as $n\to\infty$, and $\tilde{\bm\uppi}_n\in\bm{\calA}(\bm{\uppi}_n)$ for all $n$. 
Let $\tilde{\bm\uppi}_n=(\eta_{1, n}, \ldots, \eta_{N, n})$.
We fix $i\in\{1, \ldots, N\}$ and consider the
family $\{V_{i, \bm{\uppi}_n}\}$. By definition there exists $v^{i,n}\in\Usm$ such that
\begin{equation}\label{15}
\varrho^N_{i, \bm{\uppi}_n} = \int_{\calX\times\Act} \Hat{r}^N_i(y, u, \bm{\uppi}_n) v^{i,n}(du|y)\, \eta_{v^{i,n}}(dy),
\end{equation}
$\eta_{i, n}=\eta_{v^{i, n}}$, and $v^{i, n}$ is a minimizing selector of 
\begin{equation}\label{16}
V^N_{i,\bm{\uppi}_n}(x) + \varrho^N_{i, \bm{\uppi}_n} = \min_{u\in\Act}\Big\{\Hat{r}^N_i(x, u, \bm{\uppi}_n) + 
\int_{\calX} V^N_{i,\bm{\uppi}_n}(y) P(dy|x, u) \Big\}.
\end{equation}
$\eG^N$ being compact, we may assume that $\bm{\uppi}_n\to\bm{\uppi}\in\eG^N$ as $n\to\infty$. In view of (A6)
and Lemma~\ref{L5.1}, we have 
$$\Hat{r}^N_i(x, u, \bm{\uppi}_n)\to \Hat{r}^N_i(x, u, \bm{\uppi}), \quad \text{as}\; n\to\infty,$$
uniformly over compact subsets of $\calX\times\Act$.
Otherwise, we can consider a converging sub-sequence. Therefore by Remark~\ref{R-equi} we note that the family $\{V^N_{i, \bm{\uppi}_n}\}$ is locally equicontinuous and since 
$V_{i, \bm{\uppi}_n}(0)=0$ the family is also locally bounded. On the other hand, we have from (A1) that
\begin{equation*}
\int_{\calX} \calV(y)\, \eta(dy)\;\leq \; \kappa_1, \quad \text{for all}\; \eta\in\eG,
\end{equation*}
for some constant $\kappa_1>0$. Thus combining \eqref{03} and \eqref{15} we obtain that the set $\{\varrho^N_{i, \bm{\uppi}_n}, i\geq 1, n\in\NN\}$ is
compact. Now we argue to show that $V^N_{i, \bm{\uppi}_n}\in\sorder(\calV)$ uniformly in $n$. Consider the
sequence $\{K_m\}$ in \eqref{0001}. Let $m$ be such that $C\subset K_m$.
An argument similar to \eqref{a24} gives 
\begin{equation}\label{17}
V^N_{i, \bm{\uppi}_n}(x) \;=\; \inf_{v\in\Udsm}\Exp_x\Bigl[\sum_{i=0}^{\uptau(K_m)-1} (\Hat{r}^N_i(X_i, v(X_i), \bm{\uppi}_n) -
\varrho_{i, \bm{\uppi}_n}) + V^N_{i, \bm{\uppi}_n}(X_{\uptau(K_m)})\Bigr], \quad \forall\; x\in K_m^c.
\end{equation}
By \eqref{a11.51},  there exists constant $\kappa_2>0$ so that for any $x\in K^c_m$ we have
\begin{equation}\label{18}
\sup_{v\in\Usm}\,\Exp^v_x\Big[\sum_{i=0}^{\uptau(K_m)-1} \calV(X_i)\Big]\;\leq \kappa_2\;(\calV(x)+1).
\end{equation}
Since $\sup_{\bm{\uppi}_n}\sup_{u\in\Act} \Hat{r}^N_i(\cdot, u, \bm{\uppi}_n)\in\sorder(\calV)$ (by \eqref{03}) we get from \eqref{17} and \eqref{18} that
\begin{equation*}
\sup_{n}\abs{V^N_{i, \bm{\uppi}_n}}\in\sorder(\calV).
\end{equation*}
Hence along some subsequence the following holds:
\begin{equation*}
\begin{gathered}
V^N_{i, \bm{\uppi}_n}\to V^N_i, \; \text{in}\; \calC_{loc}(\calX), \quad \varrho^N_{i, \bm{\uppi}_n}\to \varrho^N_i, \quad V^N_i\in\sorder(\calV),
\\
v^{i, n}\to v^i,  \; \text{in}\; (\Usm, \D_M), \quad \norm{\eta_{v^{i,n}}-\eta_{v^i}}_{TV}\to 0.
\end{gathered}
\end{equation*}
We can pass the limit in \eqref{15} and \eqref{16} (see \eqref{14}) to obtain
\begin{align*}
V^N_{i}(x) + \varrho^N_{i} &= \min_{u\in\Act}\{\Hat{r}^N_i(x, u, \bm{\uppi}) + 
\int_{\calX} V^N_{i}(y) P(dy|x, u) \},
\\
\varrho^N_{i} &= \int_{\calX\times\Act} \Hat{r}^N_i(y, u, \bm{\uppi}) v^{i}(du|y)\, \eta_{v^{i}}(dy).
\end{align*}
Hence using Theorem~\ref{Thm-1} we have $(V^N_i, \varrho^N_i)= (V^N_{i, \bm{\uppi}}, \varrho^N_{i, \bm{\uppi}})$ and $v^i$ is a minimizing
selector of the above equation. This shows that $\eta_{v^{i}}\in\calA_i(\bm{\uppi})$. Hence the proof.
\end{proof}

Now we are ready to prove Theorem~\ref{Thm-3}.
\begin{proof}[Proof of Theorem~\ref{Thm-3}]
We note that $\bm{\uppi}\to\bm{\calA}(\bm{\uppi})$ is non-empty, compact and convex set valued (by Lemma~\ref{L4.2})
 where $\eG^N$ is a convex, compact set. We can view
$\eG^N$ as a subset of $\calM(\calX)\times\cdots\times\calM(\calX)$ which is a locally convex (with respect to weak topology)
Hausdorff space. By Lemma~\ref{L4.2}, Lemma~\ref{L4.3} and \cite[Theorem~17.10]{ali-bor} we obtain that the 
$\bm{\uppi}\to\bm{\calA}(\bm{\uppi})$ has closed graph. Thus applying Kakutani-Fan-Glicksberg fixed point theorem
\cite[Corollary~17.55]{ali-bor} we have $\bm{\uppi}\in\bm{\calA}(\bm{\uppi})$. Therefore by definition there exists 
$(v^{1}, \ldots, v^{N})\in (\Usm)^N$ so that for $\bm\uppi=(\uppi_1, \ldots, \uppi_N)$, we have $\uppi_i=\eta_{v^i}$ and
$(v^{1}, \ldots, v^{N})$ is a Nash equilibrium in the sense of Definition~\ref{defi-2}. Hence the proof.
\end{proof}

\begin{remark}\label{R4.2}
Proof of Theorem~\ref{Thm-3} shows that there exists a Nash equilibrium $(v^1, \ldots, v^N)\in(\Usm)^N$ such tha $v^i$ is a 
minimizing selector of 
\begin{equation}\label{19}
V^N_{i, \bm\uppi}(x) + \varrho^N_{i, \bm\uppi} = \min_{u\in\Act}\Big\{\Hat{r}^N_i(x, u, \bm{\uppi}) + 
\int_{\calX} V^N_{i, \bm\uppi}(y) P(dy|x, u) \Big\},
\end{equation}
and $\varrho_{i, \bm\uppi}$ is the optimal value $J^N_i(x, v, \bm{\uppi})$ defined in \eqref{9.55}. Now let us argue that any 
Nash equilibrium in $(\Usm)^N$ would be minimizing selector of \eqref{19}. Let $\bm{\uppi}=(\uppi_1, \ldots, \uppi_N)$ and 
$(v^1, \ldots, v^N)$ forms a Nash equilibrium in the sense of Definition~\ref{defi-2}. Using Theorem~\ref{Thm-1} we find solution of
\eqref{19}. By definition (see \eqref{10}), we note that $v^i$ is an optimal control and therefore by Theorem~\ref{Thm-1} we see that
$v^i$ is a minimizing selector of \eqref{19}. This also shows that $\bm{\uppi}\in\bm{\calA}(\bm{\uppi})$.
\end{remark}

Now it is easy to guess the proof of Theorem~\ref{Thm-2}. One can complete the proof by following the footsteps of Theorem~\ref{Thm-3}.
However, we add a sketch of the proof below. Reader may also like to look at \cite{anup-ari} for a similar argument
but under controlled diffusion settings.

\begin{proof}[Proof of Theorem~\ref{Thm-2}]
We consider the map $\mu\in\eG\mapsto \calA(\mu)$. By Lemma~\ref{L4.1}, $\eG$ is compact and convex and
$\eG$ is also a subset of a locally convex Hausdorff space $\calM(\calX)$. We can also mimic the arguments
of Lemma~\ref{L4.2} and Lemma~\ref{L4.3} in this set up. Rest of of the proof follows the same argument as above.
\end{proof}
\section{\bf Convergence of Nash equilibria}\label{S-convergence}

In this section we study the convergence of Nash equilibria which we find in Section~\ref{S-Nash}.
It has been shown in \cite{lasry-lions, feleqi, anup-ari}  that the $N$-person game Nash equilibrium converges to the solution of MFG. Similar results are known for controlled diffusion processes with ergodic costs. To justify 
the limit we need to have some kind of uniqueness of the limiting
equilibria. Such uniqueness is generally obtained by imposing convexity property on the Hamiltonian. To do so we define for 
$\bm{\mu}=(\mu_1, \ldots, \mu_N)\in \eG^N$, 
$$r^N_{\bm\mu}(x, u)\;\df\; \int_{(\calX)^N} r(x, u, \frac{1}{N}\sum_{j=1}^N \delta_{y_j})\, \Pi_{j=1}^N \mu_j(dy_j).$$
Consider the unique pair $(V^N_{\bm\mu}, \varrho^N_{\bm\mu})\in\calC(\calX)\cap\sorder(\calV)\times\R_+, \, V^N_{\bm\mu}(0)=0,$ that satisfies
\begin{equation}\label{200}
V^N_{\bm\mu}(x) + \varrho^N_{\bm\mu}\;=\;\min_{u\in\Act}\Bigl\{r^N_{\bm\mu}(x, u)  + \int_{\calX} V^N_{\bm\mu}(y) P(dy|x, u)\Bigr\}.
\end{equation}
Existence of such unique pair is guaranteed by Theorem~\ref{Thm-1}. We assume that
\begin{itemize}
\item[{\bf (A7)}] $\Act$ is a convex set. Moreover, for every $x\in\calX$, the following map 
$$u\mapsto r^N_{\bm\mu}(x, u)  + \int_{\calX} V^N_{\bm\mu}(y) P(dy|x, u),$$
is strictly convex, uniformly in $N$, where $V^N_{\bm\mu}$ is given by \eqref{200}.
\end{itemize}
It should be observed that ${\bm{\mu}}$ depends on $N$ and might vary with $N$.
By (A7), for every $x\in\calX$ and $\theta\in(0, 1)$ there exists a positive constant $\kappa_x$, independent of $N$, such that for all
$u, u_1\in\Act$, we have
\begin{align*}
&r^N_{\bm\mu}(x, \theta u + (1-\theta) u_1)  + \int_{\calX} V^N_{\bm\mu}(y) P(dy|x, \theta u + (1-\theta) u_1) 
\\
&\geq
\theta \Big[r^N_{\bm\mu}(x, u)  + \int_{\calX} V^N_{\bm\mu}(y) P(dy|x, u)\Big]
+ (1-\theta) \Big[r^N_{\bm\mu}(x, u_1)  + \int_{\calX} V^N_{\bm\mu}(y) P(dy|x, u_1)\Big] - \kappa_x.
\end{align*}

Our main result of this section is the following.
\begin{theorem}\label{Thm-4}
Let (A1)--(A7) hold.
Let $\{V^N_{i, \bm{\uppi}}, \varrho^N_{i, \bm{\uppi}}, \bm{\uppi}^N: 1\leq i\leq N\}$ form a Nash equilibrium in the sense that $(V^N_{i, \bm{\uppi}}, \varrho^N_{i, \bm{\uppi}}, \bm{\uppi})$ satisfy \eqref{19} and  $(v^1, \ldots, v^N)\in(\Usm)^N$ satisfies
$\eta_{v_i}=\uppi^N_i$ for all $i$. Then the following holds:
\begin{enumerate}
\item[{(i)}] the family $\{V^N_{i, \bm{\uppi}}, 1\leq i\leq N, N\geq 1\}$ is locally equicontinuous and locally bounded. $\{\varrho^N_{i, \bm{\uppi}}, \uppi^N_i, \; 1\leq i\leq N\}$
is also pre-compact in $\R\times\calP(\calX)$;
\item[{(ii)}] for any compact $K\subset \calX$, we have 
$$\sup_{1\leq i, j\leq N}\Big(\sup_{x\in K}\abs{V^N_{i, \bm{\uppi}}(x)-V^N_{j, \bm{\uppi}}(x)}+ \abs{\varrho^N_{i, \bm{\uppi}}-\varrho^N_{j, \bm{\uppi}}}
+ \norm{\uppi^N_i-\uppi^N_j}_{TV}\Big)
\to 0,$$
as $N\to\infty$;
\item[{(iii)}] any subsequential limit $(V, \varrho, \uppi)$ of $\{V^N_{i, \bm{\uppi}}, \varrho^N_{i, \bm{\uppi}}, {\uppi}^N_i\}$ forms a MFG solution
in the sense of Definition~\ref{defi-1}.
\end{enumerate}
\end{theorem}
From (A7) we note that there exists a unique continuos measurable selector \eqref{19}. In fact, this is the only minimizing selector. This follows from
the definition of minimizing selector \eqref{13} and (A7).
In view of Remark~\ref{R4.2} we see that every Nash equilibrium of $N$-person game converges to a solution of MFG. 
Under diffusion settings, convexity assumption on the Hamiltonian helps to conclude that that the optimal distributions
of Nash equilibrium converges to each other as $N\to\infty$ (see \cite{anup-ari, feleqi}). 
(A7) play similar role in our setting. Similar convexity property is also assumed in \cite{gomes-mohr-souza, gomes-2} but for finite state process.

\begin{proof}[Proof of Theorem~\ref{Thm-4}]
(i)\, Using \eqref{01} we can find constant $\kappa_1$ such that
\begin{equation}\label{21}
\int_{\calX} \calV(y)\, \eta_v(dy)\; \leq \; \kappa_1, \quad \text{for all}\; v\in\Usm.
\end{equation}
Since $h$ is inf-compact, we have $\{\uppi^N_i,\; 1\leq i\leq N\}_{N\geq 1}$ pre-compact in $\calP(\calX)$. Again by \cite[Theorem~14.0.1]{meyn-tweedie}
we have
$$\varrho^N_{i, \bm{\uppi}} = \int_{\calX} \Hat{r}^N_i(y, v^i(y), \bm{\uppi})\, \eta_{v^i}(dy).$$
Thus using \eqref{03} and \eqref{21} we have $\{\varrho^N_{i, \bm{\uppi}}, 1\leq i\leq N, N\geq 1\}$ bounded.
Now we show that $\{V^N_{i, \bm{\uppi}}\}$ is locally equicontinuous family. Fix $K\subset\calX$ compact.
$r:K\times\Act\times\calP(K)\to\R$ being continuous, we obtain that for every $K\subset\calX$ compact, the
following maps forms a equicontinuous family for $1\leq i\leq N, \, N\geq 1$,
$$(x, u)\mapsto \int_{(\calX)^N} r(x, u, \frac{1}{N}\sum_{j=1}^N \delta_{\mathfrak{P}_C(y_j)}) \Pi_{j=1}^N \nu_j(dy_j)
=\int_{(K)^N} r(x, u, \frac{1}{N}\sum_{j=1}^N \delta_{y_j}) \Pi_{j=1}^N \tilde\nu_j(dy_j),$$
where $\nu_j\in\calP(C)$. The above result is due to the fact that $\frac{1}{N}\sum_{j=1}^N \delta_{y_j}\in\calP(C)$.
Thus using Lemma~\ref{L5.1} we see that $\{\hat{r}^N_i\}$ forms a locally equicontinuous family. From 
Section~\ref{Ergodic HJB} we know that the value function $V^N_{i, \bm{\uppi}}$ is the limit of 
$V^{\alpha, N}_{i, \bm{\uppi}}(\cdot) -V^{\alpha, N}_{i, \bm{\uppi}}(0)$ where
\begin{equation*}
J^{\alpha, N}_{i, \bm{\uppi}}(x, v) \; =\; 
\Exp_x\Bigl[\sum_{j=0}^\infty \alpha^j\, \hat{r}^N_{i}(X_j(v), v(X_j), \bm{\uppi}) \Bigr],
\quad \text{and},\quad V^{\alpha, N}_{i, \bm{\uppi}}(x) \;=\; \inf_{v\in\Usm} J^{\alpha, N}_{i, \bm{\uppi}}(x, v).
\end{equation*}
By Lemma~\ref{L6.1} and Remark~\ref{R-equi}, we obtain that $\{V^{\alpha, N}_{i, \bm{\uppi}}(\cdot) -V^{\alpha, N}_{i, \bm{\uppi}}(0): 1\leq i\leq N, N\geq 1\}$ forms a locally equicontinuous family and thus $\{V^N_{i, \bm{\uppi}}\}$
is locally equicontinuous.

(ii)\, Now we consider the unique pair $(\tilde{V}^N, \tilde{\varrho}^N)\in\calC(\calX)\cap\sorder(\calV)\times\R$
that satisfies
\begin{equation}\label{22}
\tilde{V}^N(x) + \tilde{\varrho}^N \; =\; \min_{u\in\Act}\{\tilde{r}^N(x, u) + \int_{\calX} \tilde{V}^N(y)\, P(dy|x, u)\},
\end{equation}
where 
$$\tilde{r}^N(x, u)\;\df\; \int_{(\calX)^N} r(x, u, \frac{1}{N}\sum_{j=1}^N \delta_{y_j})\Pi_{i=1}^N \, \uppi_j(dy_j).$$
By Theorem~\ref{Thm-1} and \eqref{22}, $\tilde{\varrho}^N$ 
is the optimal ergodic value with running cost $\tilde{r}^N$.
We also know that the following holds.
\begin{equation}\label{23}
\varrho^N_{i, \bm{\uppi}}=\inf\Big\{\int_{\calX} \Hat{r}^N_{i, \bm{\uppi}}(y, v(y))\, \eta_v(dy): v\in\Usm\Big\},
\quad \tilde\varrho^N=\inf\Big\{\int_{\calX} \tilde{r}^N(y, v(y))\, \eta_v(dy): v\in\Usm\Big\}.
\end{equation}
Hence by \eqref{23}, we obtain
\begin{equation}\label{24}
\abs{\varrho^N_{i, \bm{\uppi}}-\tilde\varrho^N}\;\leq\; \sup_{v\in\Usm}\Big\{\int_{\calX}
\abs{\hat{r}^N_{i, \bm{\uppi}}(y, v(y))
-\tilde{r}^N(y, v(y))}\, \eta_v(dy)\Big\}.
\end{equation}
Now for any compact $K\subset\calX$, we have a constant $\kappa_2=\kappa_2(K)$ such that for 
any $N$-tuple $(y_1, \ldots, y_N)\in K\times\cdot\times K$, we have
\begin{equation}\label{25}
\eD_p(\frac{1}{N-1}\sum_{j=1}^{N-1}\delta_{y_j}, \frac{1}{N}\sum_{j=1}^{N}\delta_{y_j})
\;\leq\; \frac{\kappa_2}{N^{\frac{1}{p}}}.
\end{equation}
On the other hand, $\calP(K)$ is a compact subset of $\calP(\calX)$.
Thus using \eqref{03}, \eqref{21}, \eqref{24}, \eqref{25} and Lemma~\ref{L5.1} we see that 
\begin{equation}\label{25.5}
\sup_{1\leq i\leq N}\abs{\varrho^N_{i, \bm{\uppi}}-\tilde\varrho^N}\to 0, 
\quad \text{as}\quad N\to\infty.
\end{equation}
Let $K\subset\calX$ be compact. From \eqref{a21} we know that for any $\kappa\in(0, 1)$, there exists a constant $b_\kappa$ satisfying
\begin{equation}\label{26}
\sup_{v\in\Usm}\Big(\Exp_x^v[\calV(X_{\uptau(\B_\kappa)})] + 
\beta_1\Exp^v_x\Big[\sum_{j=0}^{\uptau(\B_\kappa)-1} \calV(X_j)\Big]\Big)
\leq \calV(x) + b_\kappa.
\end{equation}
On the other hand, by \eqref{a24} we have
\begin{equation}\label{27}
\begin{split}
V^N_{i, \bm{\uppi}}(x) & = \min_{v\in\Udsm} \Exp_x\Bigl[\sum_{j=0}^{\uptau(\B_\kappa)-1} (\Hat{r}^N_{i, \bm{\uppi}}(X_j, v(X_j))
- \varrho^N_{i, \bm{\uppi}}) + V^N_{i, \bm{\uppi}}(X_{\uptau(\B_\kappa)}) \Bigr],
\\
\tilde{V}^N(x) & = \min_{v\in\Udsm} \Exp_x\Bigl[\sum_{j=0}^{\uptau(\B_\kappa)-1} (\tilde{r}^N(X_j, v(X_j))
- \tilde\varrho^N + \tilde{V}^N(X_{\uptau(\B_\kappa)}) \Bigr].
\end{split}
\end{equation}
Therefore using \eqref{27} we obtain that
\begin{align}\label{28}
\abs{V^N_{i, \bm{\uppi}}(x)-\tilde{V}^N(x)} &\leq \sup_{v\in\Udsm}\Big|\Exp_x\Bigl[\sum_{j=0}^{\uptau(\B_\kappa)-1} (\Hat{r}^N_{i, \bm{\uppi}}(X_j, v(X_j))- \tilde{r}^N(X_j, v(X_j))+ \tilde\varrho^N- \varrho^N_{i, \bm{\uppi}}) \Big]\Big|\nonumber
\\
&\, \qquad + \sup_{y\in\B_\kappa}\abs{V^N_{i, \bm{\uppi}}(y)-\tilde{V}^N(y)}.
\end{align}
To prove (ii) it is enough to show that for any $\eps>0$ we can find $N_0$ large so that 
\begin{equation}\label{29}
\sup_{x\in K}\abs{V^N_{i, \bm{\uppi}}(x)-\tilde{V}^N(x)}\; \leq\; \eps, \quad \forall \; N\geq N_0, \quad \forall \; i\in\{1, \ldots, N\}.
\end{equation}
We use \eqref{26} and \eqref{28} to establish \eqref{29}. We may use arguments similar to (i) to conclude that the family $\{\tilde{V}^N, N\geq 1\}$ is locally equicontinuous and bounded. Therefore we can find $\kappa>0$ small so that
\begin{equation}\label{30}
\sup_{y\in\B_\kappa}\abs{V^N_{i, \bm{\uppi}}(y)-\tilde{V}^N(y)}\; \leq \; \eps/4,\quad \forall\; i\in\{1, \ldots, N\},
\end{equation}
where we use the fact that $\tilde{V}^N(0)=V^N_{i, \bm{\uppi}}(0)=0$. From \eqref{03} and \eqref{21} we note that
$\sup_{1\leq i\leq N}\sup_{u\in\Act}\Hat{r}^N_{i, \bm{\uppi}}(\cdot, u)\in\sorder(\calV)$,\, 
$\sup_{u\in\Act}\tilde{r}^N(\cdot, u)\in\sorder(\calV)$ uniformly in $N$. Thus we can find a cut-off function $\uppsi$ that vanishes outside
a compact subset of $\calX$ and
\begin{align}\label{31}
\sup_{v\in\Udsm}\Big|\Exp_x\Bigl[\sum_{j=0}^{\uptau(\B_\kappa)-1} (1-\uppsi(X_j))\big(\Hat{r}^N_{i, \bm{\uppi}}(X_j, v(X_j))- \tilde{r}^N(X_j, v(X_j))
\big) + \tilde\varrho^N- \varrho^N_{i, \bm{\uppi}})\Big]\Big|\; \leq \; \eps/4,
\end{align}
using \eqref{25.5}, \eqref{26} and,
\begin{equation}\label{32}
\sup_{v\in\Udsm}\Big|\Exp_x\Bigl[\sum_{j=m}^{\uptau(\B_\kappa)-1} \uppsi(X_j)(\Hat{r}^N_{i, \bm{\uppi}}(X_j, v(X_j))- \tilde{r}^N(X_j, v(X_j))\Big]\Big|\; \leq \; \eps/4,
\end{equation}
for large $N$ where we use Lemma~\ref{L5.1}, \eqref{25}, \eqref{26}. Thus we have \eqref{29} from \eqref{30}, \eqref{31} and \eqref{32}. The proof
of the fact
\begin{equation}\label{33.5}
\sup_{1\leq i, j\leq N}\, \norm{\uppi^N_i-\uppi^N_j}_{TV}\to 0, \quad \text{as}\quad N\to\infty,
\end{equation}
is a byproduct of (iii) below.

(iii)\, Let $(V, \varrho, \uppi)$ be any subsequential limit of $\{V^N_{i, \bm{\uppi}}, \varrho^N_{i, \bm{\uppi}}, 
{\uppi}^N_i\}$ as $N\to\infty$.
We have already seen that
$$\abs{\hat{r}^N_{i, \bm{\uppi}} - \tilde{r}^N}\to 0, \quad \abs{V-\tilde{V}^N}\to 0,$$
uniformly on every compact sets, along the subsequence of $N$. Using \eqref{26} and \eqref{27} it is easy to see that $\tilde{V}^N\in\sorder(\calV)$
uniformly in $N$ (see for example \eqref{a17} in Section~\ref{Ergodic HJB}). By Lemma~\ref{L5.1} it is easy to see that $\{\tilde{r}^N, \; N\geq 1\}$ is a family of locally equicontinuous functions. Hence 
letting $N\to\infty$ in \eqref{22} we obtain
\begin{equation}\label{33}
V(x) + \varrho \; =\; \min_{u\in\Act}\Big\{\tilde{r}(x, u) + \int_{\calX} V(y) P(dy|x, u)\Big\}
\end{equation}
where $V(0)=0, V\in\sorder(\calV)$ and $\tilde{r}^N\to \tilde{r}(x, u)$. Now we identify $\tilde{r}$. From (A7) we note that 
$$u\mapsto \tilde{r}(x, u) + \int_{\calX} V(y) P(dy|x, u)$$
is a strictly convex function. Therefore there exists a unique minimizing selector $v$. Moreover,  it is easy to see that the minimizing
selector in \eqref{22} converges to $v$. Thus from (A6) we have that $\norm{\uppi^N_i-\eta_v}_{TV}\to 0$ as $N\to\infty$. Thus one can use similar
argument as above to establish \eqref{33.5} (see also \cite{anup-ari}) using unique property of the solution of 
\eqref{33} and also the minimizing selector. We claim that
\begin{equation}\label{34}
\tilde{r}(x, u) = r(x, u, \eta_v).
\end{equation}
This would prove (iii) since by \eqref{33}
$$\varrho\;=\; \int_\calX \tilde{r}(x, v(x)) \eta_{v}(dx).$$
 But the above is a consequence of Lemma~\ref{L5.1} and Hewitt-Savage theorem (see \cite{anup-ari}, \cite{feleqi}). This can be achieved in two steps. (1) show that for any compact $K\subset \calX$,
 $$\abs{\int_{(\calX)^N} r(x, u, \frac{1}{N}\sum_{j=1}^N \del_{y_j}) \, \Pi_{j=1}^N \tilde\uppi^N_j(dy_j)-
\int_{(\calX)^N} r(x, u, \frac{1}{N}\sum_{j=1}^N\del_{y_j}) \, \Pi_{j=1}^N \tilde\eta(dy_j)}\to 0,$$
as $N\to\infty$ where $\tilde{\uppi}^N_j=(\tilde{\uppi}^N_j)_K$ and $\tilde\eta=(\eta_v)_K$. One may follow the calculations of \cite[p.~530]{feleqi} together with \eqref{extra} to achieve this. (2) Then apply Hewitt-Savage theorem to obtain that 
$$\int_{(\calX)^N} r(x, u, \frac{1}{N}\sum_{j=1}^N\del_{y_j}) \, \Pi_{j=1}^N \tilde\eta(dy_j)\to
r(x, u, \tilde\eta),$$
as $N\to\infty$. Now to get \eqref{34} apply Lemma~\ref{L5.1} and the fact that $\eD_p(\tilde\eta, \eta)\to 0$,
as $K$ increases to $\calX$.
\end{proof}

\section{\bf Proof of Theorem~\ref{Thm-1}}\label{Ergodic HJB}
In this section we prove Theorem~\ref{Thm-1}. As we mentioned earlier that existence result of \eqref{06} is well known for very general 
class of controlled processes. Reader may wish to look at \cite{ari-et-al, lerma-lasserre, bertsekas-shreve}
 for more details about these problems.
Using the method of split chain we show that the value function is continuous. 

For $\alpha\in(0, 1)$, we define the $\alpha$-discounted value function as follows. For $v\in\Uadm$,
\begin{equation}\label{a1}
J^\alpha_\mu(x, v) \; =\; \Exp_x\Bigl[\sum_{i=0}^\infty \alpha^i\, r_\mu(X_i(v), v_i(X_i)) \Bigr],
\quad \text{and},\quad V^\alpha_\mu(x) \;=\; \inf_{v\in\Uadm} J^\alpha_\mu(x, v).
\end{equation}
Then by (A2) and \cite{ari-et-al} we have that $V^\alpha_\mu$ is lsc and
\begin{equation}\label{a2}
V^\alpha_\mu(x) \; \df \; \inf_{u\in\Act}\, \big\{r_\mu(x, u) + \alpha \int_{\calX} V^\alpha_\mu(y)\, P(dy| x, u)\big\}
\end{equation}
Also a control $v\in\Udsm$ is optimal for \eqref{a1} if and only if $v(x)$ attains the infimum in \eqref{a2}, for all $x\in\calX$.

\begin{lemma}\label{L6.1}
The family of functions $\{V^\alpha_\mu, \, \al\in(0, 1)\}$ is equicontinuous on every compact subsets of $\calX$.
\end{lemma}

To prove Lemma~\ref{L6.1} we construct split chain. Let $\calY= \calX\times\calX$ and $\bar{x}=(x_1, x_2)\in\calY$.
For $\bar{u}=(u_1, u_2)\in\Act\times\Act$, we define $\bar{P}(\cdot|\bar{x}, \bar{u})\in\calP(\calY)$ as
\begin{align}\label{a3}
\bar{P}(B_1\times B_2|\bar{x}, \bar{u})\; =\; P(B_1| x_1, u_1) P(B_2| x_2, u_2), \quad B_i\in\calB(\calX).
\end{align}
Let $v\in\Usm$ be a stationary Markov control. By $\{X^k_i\}\df \{X^k_i(v)\}$ we denote the Markov chain with $X^k_0=x_k$
for $k=1,2$. We assume that $\{X^1_i\}$ and $\{X^2_i\}$ are independent of each other. In fact, this can be achieved by
constructing a
$\calY$ valued Markov chain $\bar{X}=\{X^1_i, X^2_i\}$ with transition probability given by \eqref{a3} and control $\bar{v}(y_1,y_2)\df(v(y_1), v(y_2))$.

\subsection{The pseudo-atom construction.} Now we introduce the Athreya-Ney-Nummelin construction of pseudo-atom.
Readers are referred to \cite{meyn-tweedie} for more details on such construction. We recall the measure $\nu$ and the
set $C$ from \eqref{02}. Denote $\bar{C}=C\times C$ and $\bar\Act=\Act\times\Act$.
Define $\bar{\nu}\in\calP(\bar{C})$ as $\bar{\nu}=\nu\times\nu$. It is easy to see from
\eqref{02} that 
\begin{equation}\label{a4}
\inf_{\bar{x}\in C\times C}\inf_{\bar{u}\in\Act\times\Act} \bar{P}(A|\bar{x}, \bar{u})\geq \gamma^2\,  \bar{\nu}(A),
\quad A\in\calB(\calY).
\end{equation}
To see this, denote by $A_{z_1}= \{z_2: \bar{z}\in A\}$ and $A_{z_2}= \{z_1: \bar{z}\in A\}$. These are Borel-measurable sets by \cite[Theorem~8.2]{rudin}. Then by \cite[Theorem~8.6]{rudin} we have for 
$\bar{x}\in\bar{C}$, $\bar{u}\in\bar{\Act}$,
\begin{align*}
P(A|\bar{x}, \bar{u}) &=\int_{\calX} P(A_{z_1}|x_2, u_2) P(dz_1|x_1, u_1)
\\
&\geq \gamma \int_{\calX} \nu(A_{z_1}) P(dz_1|x_1, u_1)
\\
&= \gamma  \big(\nu\times P(\cdot|x_1, u_1)\big)(A)
\\
&=\gamma \int_{\calX} P(A_{z_2}|x_1, u_1) \nu(dz_2)
\\
&\geq \gamma^2 \int_{\calX} \nu(A_{z_2}) \nu(dz_2) =\gamma^2 \bar{\nu}(A).
\end{align*}
Therefore the Markov chain $\bar{X}$ satisfies minorization condition \eqref{a4} with minorizing measure $\bar{\nu}$.
Let $\gamma_1=\frac{\gamma^2}{2}$. Let $\calY^*=\calY\times\{0, 1\}=\calX\times\calX\times\{0, 1\}$.
For $B\in\calB(\calY)$, we denote $B_0=B\times\{0\}$ and $B_1=B\times\{1\}$. For $\mu\in\calP(\calY)$
we define $\mu^*\in\calP(\calY^*)$ as follows. For $B\in\calB(\calY)$,
\begin{equation}\label{a5}
\begin{split}
\mu^*(B_0) &= (1-\gamma_1)\, \mu(B\cap \bar{C}) + \mu(B\cap\bar{C}^c),
\\[2mm]
\mu^*(B_1) &= \gamma_1\, \mu(B\cap \bar{C}).
\end{split}
\end{equation}
Clearly, $\mu^*(B_0)+\mu^*(B_1) = \mu(B)$ and if $B\subset \bar{C}^c$, then $\mu^*(B_0)=\mu(B)$.
On a suitable probability space $(\Omega^*, \mathcal{F}^*, \Prob^*)$, we define an $\calY^*$ valued Markov chain $Z_n\equiv (X^*_n, i^*_n)$, where $X^*_n=(X^{1, *}_n , X^{2, *}_n)\in\calY$, such that
\begin{itemize}
\item[{(1)}] The probability kernel of $\{Z_n\}$ is given as follows. For $\hat{z}=(\bar{z}, i)\in\calY^*$,
\begin{equation}\label{a6}
\hat{P}(d\hat{y}| \hat{z})=\left\{
\begin{array}{lll}
\bar{P}^*(d\hat{y}|\bar{z}, \bar{v}(\bar{z})) & \text{if}\; \hat{z}\in\calY_0\setminus \bar{C}_0, 
\\[2mm]
\frac{1}{1-\gamma_1}(\bar{P}^*(d\hat{y}|\bar{z}, \bar{v}(\bar{z}))- \gamma_1 \, \bar\nu^*(d\hat{y}))
& \text{if}\; \hat{z}\in\bar{C}_0,
\\[2mm]
\bar{\nu}^*(d\hat{y}) & \text{if}\; \hat{z}\in \calY_1.
\end{array}
\right.
\end{equation}
\item[{(2)}] The initial distribution is given as follows. For $B\in\calB(\calY)$,
\begin{align*}
\Prob^*(Z_0\in B_0) &= (1-\gamma_1) \, \Ind_{B\cap\bar{C}}(\bar{x}) + \Ind_{B\cap\bar{C}^c}(\bar{x}),
\\[2mm]
\Prob^*(Z_0\in\B_1) & = \gamma_1\, \Ind_{B\cap\bar{C}}(\bar{x}).
\end{align*}
\end{itemize}
It is well known that under above construction the probability laws of $\{X^*_n\}_{n\geq 1}$ and
$\{\bar{X}_n\}_{n\geq 1}$ are same \cite{meyn-tweedie}. We also observe that if the Markov chain 
$\{Z_n\}$ starts from $\calY^*\setminus \bar{C}^c\times\{1\}$ it stays in $\calY^*\setminus \bar{C}^c\times\{1\}$. The transition probability from $\bar{C}_1$ is same for all point in $\bar{C}_1$. That is why 
$\bar{C}_1$ is referred to as a pseudo-atom. Denote
$$\uptau^*\; \df\; \inf\{n> 0 : Z_n\in\bar{C}_1\}.$$
By our construction we see that the law of $\{X^{1, *}_{\uptau^*+n}\}_{n\geq 1}$ and
$\{X^{2, *}_{\uptau^*+n}\}_{n\geq 1}$ are same (see also \cite{biswas-budhiraja}).
\begin{lemma}\label{L6.2}
There exist a constant $\theta>0$, independent of $\bar{x}$ and $v\in\Usm$, satisfying
\begin{equation}\label{a7}
\Exp^*_{\hat{z}}\Big[\sum_{i=0}^{\uptau^*} \calV({X}^{1, *}_i) + \calV({X}^{2, *}_i)\Big]\; \leq \; \theta(\calV(x_1)+\calV(x_2)+ 1),
\end{equation}
for $\Hat{z}=(x_1, x_2, i)\in\calY_0\cup\bar{C}_1$.
\end{lemma}

\begin{proof}
Define $\bar{\calV}(\bar{x}) = \calV(x_1)+ \calV(x_2)$. Then from \eqref{01} we have
\begin{equation}\label{a8}
\int_{\calY} \bar{V}(\bar{y})\, \bar{P}(d\bar{y}| \bar{x}, \bar{u}) - \bar{\calV}(\bar{x}) \leq
 -\breve{\calV}(\bar{x}) + \kappa\, \Ind_{\bar C}(\bar{x}),
\end{equation}
for some constant $\kappa>0$, where $\breve{\calV}(\bar{x})=\frac{\beta_1}{4}(\calV(x_1) + \calV(x_2))$.
Define 
\begin{align*}
\bar\uptau & = \inf\{ n\geq 1\;:\; \bar{X}_n\in\bar{C}\},
\\
\uptau_1 &= \inf\{n\geq 1 \; :\; X^{*}_n\in\bar{C} \}.
\end{align*} 
By above property of split Markov chain we have
$$\bar{E}_{\bar{x}}[\bar\uptau] \; =\; \Exp^*_{\delta^*_{\bar{x}}} [\uptau_1],$$
where $\delta^*_{\bar{x}}$ is defined as in \eqref{a5}. Using \eqref{a8} we can find constant
$\kappa_1$ satisfying (see for example, \cite[Theorem~14.2.2]{meyn-tweedie})
$$\bar{E}_{\bar{x}}\Big[\sum_{i=0}^{\bar\uptau-1}\breve{\calV}(\bar{X}_i)\Big]
\;\leq\; \kappa_1(\bar{\calV}(\bar{x}) + 1).$$
Thus for any $\hat{z}=(\bar{z}, i)\in \calY_0\cup \bar{C}_1$ we have from above that
\begin{equation}\label{a9}
\Exp^*_{\hat{z}} \Big[\sum_{i=0}^{\uptau_1-1}\breve{\calV}({X}^*_i)\Big]\; \leq\; \frac{\kappa_1}{\gamma_1\wedge(1-\gamma_1)}(\bar{\calV}(\bar{z}) + 1).
\end{equation}
Now define a sequence of stopping times $\{\vartheta_k\}$ as $\vartheta_0=0$ and
$\vartheta_k=\inf\{n>\vartheta_{k-1}\; :\; X^*_n\in\bar{C}\}$. By \eqref{a9} we have 
$\Prob^*_{\hat{z}}(\vartheta_k<\infty)=1$ for all $k\in\NN$ and $\hat{z}\in\calY_0\cup \bar{C}_1$.
Denote by $\mathfrak{F}_n = \sigma\{Z_i\; :\; i\leq n\}$. Then
\begin{align}\label{a10}
\Prob^*_{\hat{z}}(\uptau^*>\vartheta_k) & = \Prob^*_{\hat{z}}(\uptau^*>\vartheta_k, \, \uptau^*>\vartheta_{k-1})\nonumber
\\
&= \Exp^*\Big[\Ind_{\{\uptau^*>\vartheta_{k-1}\}} \Prob^*_{\hat{z}}
(\uptau^*>\vartheta_k|\mathfrak{F}_{\vartheta_{k-1}})
\Big]\nonumber
\\
&\leq \Exp^*\Big[\Ind_{\{\uptau^*>\vartheta_{k-1}\}} \sup_{\hat{z}\in\bar{C}_0}\Prob^*_{\hat{z}}
(\uptau^*> 1)\Big]\nonumber
\\
&\leq [\sup_{\hat{z}\in\bar{C}_0}\Prob^*_{\hat{z}}(\uptau^*> 1)]^{k-1}.
\end{align}
For $\bar{z}\in \bar{C}$ we have from \eqref{a4}, \eqref{a6} that
\begin{align*}
\Prob^*_{(\bar{z}, 0)}(\uptau^*=1)= \frac{1}{1-\gamma_1} (\gamma_1\bar{P}(\bar{C}|\bar{z}, v(\bar{z}))
- \gamma_1^2 \bar{\nu}(\bar{C}))\geq \frac{\gamma_1^2}{(1-\gamma_1)}\df \gamma_2\in(0, 1).
\end{align*}
Hence \eqref{a10} gives us that
\begin{align}\label{a11}
\Prob^*_{\hat{z}}(\uptau^*>\vartheta_k)\leq (1-\gamma_2)^{k-1}.
\end{align}
Letting $k\to\infty$ in \eqref{a11} we find that $\Prob^*_{\hat{z}}(\uptau^*<\infty)=1$.
Let 
$$\calH(\vartheta_k)\;\df\; \sum_{i=0}^{\vartheta_k-1}\breve{\calV}({X}^*_i).$$
Then for $\hat{z}=(\bar{z}, i)\in \calY_0\cup \bar{C}_1$ we get
\begin{align*}
\Exp^*_{\hat{z}}[\sum_{i=0}^{\uptau^*-1}\breve{\calV}({X}^*_i)] 
&= \Exp^*_{\hat{z}}\Big[\sum_{k=0}^\infty \Ind_{\{\uptau^*=\vartheta_k\}} \calH(\vartheta_k)\Big]\\
&= \Exp^*_{\hat{z}}\Big[\sum_{k=0}^\infty \Ind_{\{\uptau^*=\vartheta_k\}} \sum_{l=0}^{k-1}(\calH(\vartheta_{l+1})-\calH(\vartheta_{l}))\Big]
\\
&= \Exp^*_{\hat{z}}\Big[\sum_{l=0}^\infty  \Ind_{\{\uptau^*\geq \vartheta_{l+1}\}} (\calH(\vartheta_{l+1})-\calH(\vartheta_{l}))\Big]
\\
&\leq \kappa_3(1+\calV(\bar{z})) + \Exp^*_{\hat{z}}\Big[\sum_{l=1}^\infty \Ind_{\{\uptau^*> \vartheta_{l}\}} 
(\calH(\vartheta_{l+1})-\calH(\vartheta_{l}))\Big]
\\
&\leq \kappa_3(1+\calV(\bar{z})) + 
\Exp^*_{\hat{z}}\Big[\sum_{l=1}^\infty \Ind_{\{\uptau^*> \vartheta_{l}\}} \sup_{\hat{z}\in\bar{C}\times\{0,1\}}\Exp^*_{\hat{z}}[\calH(\vartheta_1)]\Big]
\\
&\leq \kappa_3(1+\calV(\bar{z})) + \kappa_4\sum_{l\geq 1}(1-\gamma_2)^{l-1}
\\
&\leq \kappa_5 (1+\calV(\bar{z})),
\end{align*}
for some constants $\kappa_3, \kappa_4, \kappa_5$, where in the sixth line
we use \eqref{a9}, \eqref{a11}. This proves \eqref{a7} since $\breve{\calV}$ is bounded on $\bar{C}$.
\end{proof}
Next Lemma establishes moment estimate of $\uptau^*$.
\begin{lemma}\label{L6.3}
For any compact set $K\subset\calX$ we can find positive constants $\del_1, \, \kappa_K$
such that
$$\sup_{v\in\Usm}\Exp^*_{\hat{z}}[e^{\del_1\, \uptau^*}]\; \leq\; \kappa_K \quad \text{for all}\; \hat{z}\in 
K_0\cup C_1.$$
\end{lemma}

\begin{proof}
Applying \cite[Theorem~15.2.5]{meyn-tweedie} and \eqref{a8} we can find $\del_2>0$ such that 
\begin{equation}\label{a11.50}
\sup_{v\in\Usm}\bar{\Exp}_{\bar{z}}[e^{\del_2\, \uptau}]\;\leq\; \kappa_1(\bar{\calV}(\bar{z}) +1),
\end{equation}
for some constant $\kappa_1$. Therefore from the property of Markov chain and \eqref{a11.50} can choose $\del_3$ small enough to satisfy 
$$\sup_{\bar{z}\in \bar{C}}\, \sup_{v\in\Usm}\,\Exp^*_{(\bar{z}, i)} [e^{\del_3 \uptau_1}]\;<\; \frac{1}{1-\gamma_1}.$$
Then the result follows from \cite[Proposition~4]{masi-stettner} (see also \cite[Lemma~4.7]{biswas-budhiraja}).
\end{proof}
The next lemma will be useful to show that the solution is in $\sorder(\calV)$.
\begin{lemma}\label{L6.4}
For any compact $K\supset C$, The function $x\mapsto \sup_{v\in\Usm}\Exp^v_x[\uptau(K)]$ is in 
$\sorder(\calV)$. 
\end{lemma}

\begin{proof}
Define $G(x)\df \sup_{v\in\Usm}\Exp^v_x[\uptau(K)]$. Thus using \eqref{01} we have (see \cite[Theorem 14.2.2]{meyn-tweedie})
\begin{equation}\label{a11.51}
\sup_{v\in\Usm}\Exp^v_x[\sum_{i=0}^{\uptau(K)-1} \calV(X_i)]\leq \kappa_1(\calV(x)+1),
\end{equation}
for some constant $\kappa_1$. In fact, the constant $\kappa_1$ is independent of any compact $K$ containing $C$.
Therefore we only need to consider the case when $\calX$ is not compact. 
Let $\eps>0$. Consider $K_n\supset K$ and let $x\in K^c_n$. Then we have for all $v\in\Usm$
\begin{align*}
\Exp^v_x[\uptau(K)] &= \Exp^v_x[\uptau(K)-\uptau(K_n)] + \Exp^v_x[\uptau(K_n)]
\\
&\leq \sup_{x\in K_n}\Exp^v_{x}[\uptau(K)] + \frac{1}{\inf_{K_n^c}\calV} \Exp^v_x[\sum_{i=0}^{\uptau(K_n)-1} \calV(X_i)]
\\
&\leq \sup_{x\in K_n}\Exp^v_{x}[\uptau(K)] + \frac{1}{\inf_{K_n^c}\calV} \kappa_1(\calV(x)+1).
\end{align*}
Therefore we can choose $n$ large enough to satisfy the following
$$G(x)\leq \kappa_\eps + \eps\, (\calV(x)+1), \quad \text{for all}\; x\in\calX.$$
$\eps$ being arbitrary we get from the above expression that $G\in\sorder(\calV)$.
\end{proof}

Now we are ready to prove Lemma~\ref{L6.1}.
\begin{proof}[Proof of Lemma~\ref{L6.1}]
Consider a compact subset $K$ of $\calX$ and $x_1, x_2\in K$. Let $v^\alpha\in\Udsm$ be a minimizer selector of \eqref{a2}. Then we have
$$V^\alpha_\mu(x)= \Exp_x\Big[\sum_{i=0}^\infty \alpha^i\, r_\mu(X_i, v^\alpha(X_i))\Big],$$
for all $x$. Therefore 
$$V^\alpha_\mu(x)=\min_{v\in\Usm} J(x, v).$$
For $v\in\Usm$, define
$$\Tilde{J}(x, v)\df \Exp_x\Bigl[\sum_{i=1}^\infty \al^i r_\mu(X_i, v(X_i))\Bigr].$$
Hence
\begin{equation}\label{a11.5}
\abs{V^\alpha_\mu(x_1)-V^\alpha_\mu(x_2)} \leq \max_{u\in\Act}\, \abs{r_\mu(x_1, u)-r_\mu(x_2, u)}
+ \sup_{v\in\Usm}\abs{\tilde{J}(x_1, v)-\tilde{J}(x_2, v)}.
\end{equation}
Thus to show equicontinuity we only need to show that the second term on the rhs of \eqref{a11.5} is small
whenever $\D_\calX(x_1, x_2)$ is small.
Let $\{K_n\}$ be the increasing sequence of compact subsets in $\calX$ chosen in \eqref{0001}.  
Let $\uppsi_n:\calX\to[0,1]$ be a Lipschitz
continuous function with property that $\uppsi_n= 1$ on $K_n$ and vanishes on $K^c_{n+1}$. Consider $\eps>0$. Let $v\in\Usm$. Then
we can choose $n$ large enough so that
\begin{align}\label{a12}
&\Exp_{x_1}\Big[\sum_{i=1}^{\infty} \alpha^i (1-\uppsi_n(X_i))r_\mu(X_i, v(X_i)) \Big]
- \Exp_{x_2}\Big[\sum_{i=1}^{\infty} \alpha^i (1-\uppsi_n(X_i))r_\mu(X_i, v(X_i)) \Big]\nonumber
\\
&= \Exp^*_{\delta^*_{\bar{x}}}\Big[\sum_{i=1}^{\infty} \alpha^i (1-\uppsi_n(X^{1, *}_i))r_\mu(X^{1, *}_i, v(X^{1, *}_i)) 
-  \alpha^i (1-\uppsi_n(X^{2, *}_i))r_\mu(X^{2, *}_i, v(X^{2, *}_i)) \Big]\nonumber
\\
&=\Exp^*_{\delta^*_{\bar{x}}}\Big[\sum_{i=1}^{\uptau^*} \alpha^i \Big((1-\uppsi_n(X^{1, *}_i))r_\mu(X^{1, *}_i, v(X^{1, *}_i)) 
-  (1-\uppsi_n(X^{2, *}_i))r_\mu(X^{2, *}_i, v(X^{2, *}_i))\Big) \Big]\nonumber
\\
&\leq 4\, \sup_{x\in K^c_{n}}\frac{\sup_{u\in\Act} r_\mu(x)}{\calV(x)} \Exp^*_{\delta^*_{\bar{x}}}\Big[\sum_{i=1}^{\uptau^*}  \breve{\calV}(X^{*}_i) \Big]\nonumber
\\
&<\eps/4,
\end{align}
where in the last line we use Lemma~\ref{L6.2}. We fix $n$ from above. For $k\in\NN$, we calculate
\begin{align*}
&\Exp_{x_1}\Big[\sum_{i=k}^{\infty} \alpha^i\, \uppsi_n(X_i)r_\mu(X_i, v(X_i)) \Big]
- \Exp_{x_2}\Big[\sum_{i=k}^{\infty} \alpha^i\, \uppsi_n(X_i)r_\mu(X_i, v(X_i)) \Big]
\\
& = \Exp^*_{\delta^*_{\bar{x}}}\Big[\sum_{i=k}^{\infty} \Ind_{\{k\leq \uptau^*\}}\alpha^i 
\Big(\uppsi_n(X^{1, *}_i)r_\mu(X^{1, *}_i, v(X^{1, *}_i)) 
-  \uppsi_n(X^{2, *}_i)r_\mu(X^{2, *}_i, v(X^{2, *}_i)) \Big)\Big]
\\
& = \Exp^*_{\delta^*_{\bar{x}}}\Big[\sum_{i=k}^{\uptau^*} \Ind_{\{k\leq \uptau^*\}}\alpha^i 
\Big(\uppsi_n(X^{1, *}_i)r_\mu(X^{1, *}_i, v(X^{1, *}_i)) 
-  \uppsi_n(X^{2, *}_i)r_\mu(X^{2, *}_i, v(X^{2, *}_i)) \Big)\Big]
\\
&\leq \kappa_n \Exp^*_{\delta^*_{\bar{x}}}[(\uptau^*-k)^+],
\end{align*}
where $\kappa_n\df 2\,\sup_{(x, u)\in K_{n+1}\times\Act}\, r_\mu(x, u)$. Therefore using Lemma~\ref{L6.3} we can choose $k$ large enough so that
\begin{align}\label{a13}
\Exp_{x_1}\Big[\sum_{i=k}^{\infty} \alpha^i\, \uppsi_n(X_i)r_\mu(X_i, v(X_i)) \Big]
- \Exp_{x_2}\Big[\sum_{i=k}^{\infty} \alpha^i\, \uppsi_n(X_i)r_\mu(X_i, v(X_i)) \Big]\leq \eps/4.
\end{align}
Fix $k$ as choosen above. Note that the above choice of $n$ and $k$ is independent of $v\in\Usm$.
 Now using (A2) it is clear that if $\D_\calX(x_1, x_2)$ is small enough then we have
\begin{align}\label{a14}
\Big|\Exp^v_{x_1}\Big[\sum_{i=1}^{k} \alpha^i\, \uppsi_n(X_i(v))r_\mu(X_i, v(X_i)) \Big]
- \Exp^v_{x_2}\Big[\sum_{i=1}^{k} \alpha^i\, \uppsi_n(X_i(v))r_\mu(X_i, v(X_i)) \Big]\Big|\leq \eps/4,
\end{align}
for all $v\in\Usm$.
Thus combining \eqref{a11.5}, \eqref{a12}, \eqref{a13} and \eqref{a14} we see that there exists $\delta>0$ such that 
$$\abs{V^\alpha_\mu(x_1)-V^\alpha_\mu(x_2)}<\eps, \quad \text{whenever}\quad \D_\calX(x_1, x_2)\leq \delta, \; x_1, x_2\in K.$$
Hence the proof.
\end{proof}

\begin{remark}\label{R-equi}
In fact, the proof Lemma~\ref{L6.1} shows that if $\calK\subset\calP(\calX)$ is compact then the family $\{V^\alpha_\mu, \alpha\in(0, 1), \mu\in\calK\}$ is equicontinuous on every compact subsets of $\calX$. To see this, we observe that  the estimate \eqref{a12} holds uniformly
in $\mu\in\calK$ by (A5)(1). Again for any compact set $K_1\subset\calX$ we have 
$$\sup_{(x, u)\in K_1\times\Act}\abs{r_\mu(x, u)-r_{\mu_n}(x, u)}\to 0, \quad \text{if}\quad \eD_p(\mu, \mu_n)\to 0.$$
Thus using (A2) we can have \eqref{a14} uniformly in $\mu\in\calK$.
\end{remark}
Define $\bar{V}^\alpha_\mu(x) \df V^\alpha(x)-V^\alpha(0)$. Thus by Lemma~\ref{L6.1} we obtain $\{\bar{V}^\alpha_\mu, \, \alpha\in(0, 1)\}$ locally bounded and locally equicontinuous. On the other hand, from \eqref{01} we have
for $\alpha\in(0,1)$ and  $v\in\Uadm$ that
\begin{align*}
\Exp^v_x[\alpha^k\calV(X_k)]- \calV(x) & = \sum_{i=1}^k \alpha^i \Exp_x^u[\calV(X_i)]-\alpha^{i-1}\Exp_x^u[\calV(X_{i-1})]
\\
&\leq \sum_{i=1}^k \alpha^{i-1} (\Exp_x^u[\calV(X_i)]-\Exp_x^u[\calV(X_{i-1})])
\\
&\leq \sum_{i=1}^k \alpha^{i-1} \Big(\beta_2\Ind_{C}(X_{i-1})-\beta_1\Exp_x^u[h(X_{i-1})]\Big).
\end{align*}
Therefore letting $k\to\infty$, and using the fact that $\calV\geq 1$, we get
\begin{equation}\label{a15}
\beta_1\sum_{i=0}^\infty\, \alpha^i\, \Exp_x^u[h(X_{i})] \; \leq \; \calV(x) + \frac{\beta_2}{1-\alpha}.
\end{equation}
Since $\max_{u\in\Act}\, r_\mu(\cdot, u)\in\sorder(\calV)$, we have from \eqref{a15} that for every compact $K\subset\calX$
\begin{equation}\label{a15.5}
(1-\alpha)\, V^\alpha_\mu(x) \leq \kappa_1,
\end{equation}
for some constant $\kappa_1$, depending on $K$. Therefore we can extract a subsequence of $\{\alpha\}$ tending
to $1$ so that as $\alpha\to 0$,
\begin{align*}
\bar{V}^\alpha_\mu \to V_\mu, \quad \text{and},\quad (1-\alpha)V^\alpha_\mu(x)\to \varrho, \; \forall\; x\in\calX.
\end{align*}
Now we intend to pass the limit in \eqref{a2}. To justify the limit we need to it is enough to show that
$\bar{V}^\alpha_\mu\in\sorder(\calV) $,uniformly in $\alpha$.
This can be achieved by Lemma~\ref{L6.4}. Let $\eps>0$ be given. We consider the sequence $\{K_n\}$ from
\eqref{0001}. Fix some $n$ with $K_n\supset C$, and let $x\in K_n^c$. Since $v^\alpha$ is a minimising selector, we get
\begin{align*}
\abs{\bar{V}^\alpha_\mu(x)} &= \Big|\Exp_x\Big[\sum_{i=0}^\infty\alpha^i r_\mu(X_i, v^\alpha(X_i))\Big]
- V^\alpha_\mu(0)\Big|
\\
&\leq \Big|\Exp_x\Big[\sum_{i=0}^{\uptau(K_n)-1}\alpha^i r_\mu(X_i, v^\alpha(X_i))\Big]\Big| +
\Big|\Exp_x\Big[\alpha^{\uptau(K_n)}V^\alpha_\mu(X_{\uptau(K_n)})\Big]
- V^\alpha_\mu(0)\Big|
\\
&\leq \sup_{K^c_n}\frac{\sup_{u\in\Act}r_\mu(x, u)}{\calV(x)}\;\Exp_x\Big[\sum_{i=0}^{\uptau(K_n)-1}\alpha^i \calV(X_i)\Big] +
\Big|\Exp_x\Big[(\alpha^{\uptau(K_n)}-1)V^\alpha_\mu(X_{\uptau(K_n)})\Big]\Big|
\\
&\quad +
\Big|V^\alpha_\mu(X_{\uptau(K_n)})- V^\alpha_\mu(0)\Big|
\\
&\leq \sup_{K^c_n}\frac{\sup_{u\in\Act}r_\mu(x, u)}{\calV(x)} \kappa_1(\calV(x)+1) + \Exp_x[\uptau(K_n)]\,
\sup_{x\in K_n}(1-\alpha)V^\alpha_\mu(x) + \sup_{x\in K_n}\bar{V}^\alpha_\mu(x),
\end{align*}
where in third inequality we use \eqref{a11.51} and the inequality $(1-z^m)\leq m(1-z)$ for all $z\in[0, 1], \, m\in\NN$.
Now choose $n$ large enough and apply \eqref{a15.5} to get
\begin{equation}\label{a17}
\abs{V_\mu(x)}\; \leq \kappa_\eps + \eps\calV(x) + \sup_{v\in\Usm}\Exp_x[\uptau(K_n)], \quad \text{for all}\; x\in \calX. 
\end{equation}
$\eps$ being arbitrary, applying Lemma~\ref{L6.4} we have from \eqref{a17} that $V^\alpha_\mu\in\sorder(\calV)$
uniformly in $\alpha\in(0, 1)$. This shows $V_\mu\in\sorder(\calV)$.
Thus we obtain that every sub sequential limit of $\{\bar{V}^\alpha_\mu\}$ satisfies
\begin{equation}\label{a16}
V_\mu(x) + \varrho =\min_{u\in\Act}\{r_\mu(x, u) + \int_{\calX} V_\mu(y)\, P(dy|x, u)\}.
\end{equation}
Also an application of Tauberian theorem shows that
$$\varrho\; \leq \; \varrho_\mu,$$
where $\varrho_\mu$ is given by \eqref{05}.  
From \eqref{01} we obtain that
$$\limsup_{n\to\infty} \frac{1}{n} \sup_{v\in\Usm}\Exp^v_{x}[\calV(X_n)]\; \leq \frac{\beta_2}{\beta_1}.$$
Since $V_\mu\in\sorder(\calV)$ (or equivalently, using \eqref{a17}) we have
\begin{equation}\label{a18}
\limsup_{n\to\infty} \frac{1}{n} \sup_{v\in\Usm}\Exp^v_{x}[\abs{V_\mu(X_n)}]=0, \quad \text{for all}\; x\in\calX.
\end{equation}
Let $\tilde{v}\in\Udsm$ be a minimizing selector of \eqref{a16} . For existence of measurable selector one may refer
\cite[Proposition~7.33]{bertsekas-shreve}.
Then for any $n\in\NN$ we have
$$V_\mu(x) + n\varrho = \Exp_x^{\tilde{v}}\Bigl[\sum_{i=0}^{n-1} r_\mu(X_i, \tilde{v}_i(X_i))\Bigr] + \Exp_x^{\tilde{v}}[V_\mu(X_{n})].$$
Now dividing both sides by $n$ and using \eqref{a18} we get
$$\varrho = \lim_{n\to\infty}\;\frac{1}{n}\Exp_x^{\tilde{v}}\Bigl[\sum_{i=0}^{n-1} r_\mu(X_i, \tilde{v}_i(X_i))\Bigr].$$
Thus
\begin{equation}\label{a19}
\varrho=\varrho_\mu=\lim_{n\to\infty}\;\frac{1}{n}\Exp_x^{\tilde{v}}\Bigl[\sum_{i=0}^{n-1} r_\mu(X_i, \tilde{v}_i(X_i))\Bigr], \quad \forall\; x\in\calX.
\end{equation}
Now we are ready to prove Theorem~\ref{Thm-1}.
\begin{proof}[Proof of Theorem~\ref{Thm-1}]
Existence of $(V_\mu, \varrho)$ with $V_\mu\in\sorder(\calV),\,\varrho=\varrho_\mu,$ is shown in \eqref{a16} and \eqref{a19}. Now consider any
pair $(V, \varrho)\in\calC(\calX)\times\R$ with $V(0)=0, \, V\in\sorder(\calV)$, that satisfies
\begin{equation}\label{a20}
V(x) + \varrho =\min_{u\in\Act}\{r_\mu(x, u) + \int_{\calX} V(y)\, P(dy|x, u)\}.
\end{equation}
It is easy to see that $V$ satisfies \eqref{a18}. Hence an argument similar to above can be used to show that $\varrho=\varrho_\mu$. It is enough to
show that $V=V_\mu$ where $V_\mu$ obtained as a subsequential limit of $\{\bar{V}^\alpha_\mu\}$.
Consider a point
$x_0\in C$ that lies in the support on $\nu$ (see \eqref{01}). Without loss of generality we assume that $x_0=0$. Define 
\begin{equation}\label{a20.5}
\B_\kappa\;\df\;\{x: \D_\calX(0, x)\leq \kappa\}, \quad a_\kappa\;\df\; \nu(\B_\kappa).
\end{equation}
$0$ being in the support of $\nu$ we have $a_\kappa>0$ for all $\kappa>0$. 
Now we recall that
$\uptau(A)$ denotes the hitting/return time to $A$. Since by \eqref{02} we have 
$$\Ind_C(x)\;\leq \; \frac{1}{\gamma\, a_\kappa} \inf_{u\in\Act}P(\B_\kappa| x, u),$$
we get from \eqref{01} that for any $v\in\Usm$,
\begin{align}\label{a21}
\Exp^v_x[\calV(X_{\uptau(\B_\kappa)})] + \beta_1\Exp^v_x\Big[\sum_{i=0}^{\uptau(\B_\kappa)-1} \calV(X_i)\Big] &\leq \calV(x) + 
\beta_2\Exp^v_x\Big[\sum_{i=0}^{\uptau(\B_\kappa)-1}\Ind_C(X_i)\Big]\nonumber
\\
&\leq \calV(x) + \frac{\beta_2}{\gamma a_\kappa}\Exp^v_x\Big[\sum_{i=0}^{\uptau(\B_\kappa)-1}P(\B_\kappa|X_{i}, v(X_i))\Big]\nonumber
\\
&\leq \calV(x) + \frac{\beta_2}{\gamma a_\kappa}.
\end{align}
One can use similar argument as in \eqref{a21} to obtain that
\begin{equation}\label{a22}
\limsup_{n\to\infty}\sup_{v\in\Usm}\Exp^v_x[\calV(X_{n\wedge\uptau(\B_\kappa)})]<\infty, \quad \forall\; \kappa>0.
\end{equation}
Since $V\in\sorder(\calV)$, using \eqref{a22} we obtain that for any
$v\in\Usm$ 
\begin{align}\label{a23}
\lim_{n\to\infty} \Exp_x[V(X_{n\wedge\uptau(\B_\kappa)})] &= \lim_{n\to\infty} \Exp_x[\Ind_{\{n\geq \uptau(\B_\kappa)\}}V(X_{\uptau(\B_\kappa)})] +
\lim_{n\to\infty} \Exp_x[\Ind_{\{n< \uptau(\B_\kappa)\}}V(X_{n})]\nonumber
\\
&=  \Exp_x[V(X_{\uptau(\B_\kappa)})].
\end{align}
Since there is a minimizing selector of \eqref{a20} in $\Udsm$ we get  from \eqref{a21}, \eqref{a23} and dominated convergence theorem that
\begin{equation}\label{a24}
V(x) \;=\; \inf_{v\in\Udsm}\Exp_x\Bigl[\sum_{i=0}^{\uptau(\B_\kappa)-1} (r_\mu(X_i, v(X_i)) -\varrho_\mu) + V(X_{\uptau(\B_\kappa)})\Bigr],
\quad \forall\; \kappa>0.
\end{equation}
The infimum in \eqref{a24} can be replaced by minimum since it is achieved by the minimizer selector of \eqref{a20}.
It is worth mentioning that similar expression of value function is known for diffusion control problems \cite[Chapter~3.7]{ari-bor-ghosh}. We observe that 
$V_\mu$ also satisfies equation similar to \eqref{a24}. Now let $\tilde{v}$ be a minimizing selector of \eqref{a20}. Then $\tilde{v}$ would be sub-optimal for
$V_\mu$. Hence from \eqref{a24} we would get
$$V_\mu(x)\leq V(x) + \sup_{\B_\kappa}V_\mu - \inf_{\B_\kappa} V, \quad \forall\; \kappa>0.$$
Since $V(0)=V_\mu(0)=0$, letting $\kappa\to 0$, we get that $V_\mu\leq V$. Similarly, we also obtain $V\leq V_\mu$ and thus $V=V_\mu$.

Now we are remained to show that every optimal control $v\in\Usm$ is a minimizing selector of \eqref{a20} in the sense that
\begin{equation}\label{a25}
V_\mu(x) + \varrho_\mu = r_\mu(x, v(x)) + \int_{\calX} V_\mu(y) P(dy|x, v(x)), \quad \text{almost surely with respect to}\; \eta_v. 
\end{equation}
Let \eqref{a25} does not hold. 
Since the rhs of the above display is locally finite we can find non-negative $f\neq 0$ in $L^\infty(\eta_v)$, supported in some $\B_\kappa$ ,
that safisfies
\begin{equation}\label{a26}
V_\mu(x) + \varrho_\mu + f(x) = r_\mu(x, v(x)) + \int_{\calX} V_\mu(y) P(dy|x, v(x)).
\end{equation}
Then using \eqref{a26} and an argument as in \eqref{a19} we have
\begin{align*}
\varrho_{\mu} &= \lim_{n\to\infty} \frac{1}{n}\, \Exp_x\Big[\sum_{i=0}^{n-1} \Big(r_{\mu}(X_i, v(X_i))-f(X_i)\Big)\Big] 
\\
&= \lim_{n\to\infty} \frac{1}{n}\, \Exp_x\Big[\sum_{i=0}^{n-1} r_{\mu}(X_i, v(X_i))\Big] - \int_{\calX} f(y)\, \eta_v(dy)
\\
&=\varrho_\mu- \int_{\calX} f(y)\, \eta_v(dy),
\end{align*}
by \cite[Theorem~14.3.3]{meyn-tweedie} and the optimality of $v$. But this implies $f=0$ almost surely with respect to $\eta_v$ which is a contradiction.
This proves \eqref{a25}.
\end{proof}


\bibliographystyle{plain}

\bibliography{MFG}

\end{document}